\documentclass[reqno, 11pt]{amsart}
\usepackage{amssymb, amsmath, amsthm, mathrsfs, amsfonts, latexsym}
\usepackage[margin=1in]{geometry}

\newtheorem{theorem}{Theorem}[section]
\newtheorem{lemma}[theorem]{Lemma}

\newtheorem{proposition}[theorem]{Proposition}
\newtheorem{corollary}[theorem]{Corollary}
\newtheorem{remark}[theorem]{Remark}

\theoremstyle{definition}

\numberwithin{equation}{section}

\def\R{{\mathbb R}}

\def\eps{\varepsilon}

\def\E{{\mathbb E}}
\def\P{{\mathbb P}}
\def\Z{{\mathbb Z}}

\begin{document}

\title[Local Nondeterminism and Local times of the stochastic wave equation]
{Local nondeterminism and local times of the stochastic wave equation driven by 
fractional-colored noise}
\author{Cheuk Yin Lee}
\address{Institut de math\'ematiques, \'Ecole polytechique f\'ed\'erale de Lausanne,
Station 8, CH-1015 Lausanne, Switzerland}
\email{cheuk.lee@epfl.ch}



\begin{abstract}
We investigate the existence and regularity of the local times of 
the solution to a linear system of stochastic wave equations 
driven by a Gaussian noise that is fractional in time and colored in space.
Using Fourier analytic methods, we establish strong local nondeterminism properties of 
the solution and the existence of jointly continuous local times. 
We also study the differentiability and moduli of continuity of the local times and deduce 
some sample path properties of the solution.
\end{abstract}

\maketitle

\section{Introduction}

The stochastic wave equation is a fundamental 
stochastic partial differential equation (SPDE) of hyperbolic type.
The wave equation driven by the space-time white noise and 
temporally-white, spatially-correlated noises has been studied considerably
in the literature; see e.g.\ \cite{C70, CN88, DF98, DS09, MS99, M97, PZ00, W86}.
Since the recent development of stochastic calculus with respect to fractional
Brownian motion, there have been growing interests in the studies of 
SPDEs driven by fractional Brownian motion and other
fractional Gaussian noises in time and/or in space, for which we refer, among others, to 
\cite{BC16, BJQ15, BT08, CHKN18, HHNT15, HSWX20, MN03, NV06, QT07}.

In this paper, we study the stochastic wave equation driven by a Gaussian noise 
that is fractional with Hurst index $1/2 < H < 1$ in time (or white in time), and is
colored in space with some spatial covariance.
More precisely, we consider the following system of linear stochastic wave equations
\begin{equation}\label{SWE}
\begin{split}
\begin{cases}
\displaystyle\frac{\partial^2}{\partial t^2} u_j(t, x) = \Delta u_j(t, x) + \dot W_j(t, x)
& \text{for } t \ge 0, x \in \R^N,\\
u_j(0, x) = 0, \quad \displaystyle\frac{\partial}{\partial t}u_j(t, 0) = 0,
& j = 1, \dots, d,
\end{cases}
\end{split}
\end{equation}
where $\Delta$ is the Laplacian in variable $x$
and $\dot W = (\dot W_1, \dots, \dot W_d)$ is a $d$-dimensional Gaussian noise.
We assume that $\dot W_1, \dots, \dot W_d$ are i.i.d.\
and formally, each $\dot W_j$ has covariance
\begin{equation}\label{Wcov}
\E[\dot W_j(t, x) \dot W_j(s, y)] = \rho_H(t-s) f_\beta(x-y),
\end{equation}
where
\[ \rho_H(t-s) = \begin{cases}
\delta(t-s) & \text{if }H = 1/2,\\
|t-s|^{2H-2} & \text{if }1/2 < H < 1
\end{cases} \]
and the spatial covariance function $f_\beta$ is the Fourier transform of a measure
$\mu_\beta$ which has a density $h_\beta(\xi)$ (with respect to the Lebesgue measure) 
that is comparable to $|\xi|^{-\beta}$, where $0 < \beta < N$.

For the stochastic wave equation driven by the space-time white noise,
there is no function-valued solution when the spatial dimension $N$ is greater than 1.
An approach to study this equation in higher dimensions is to consider
a Gaussian noise that is white in time but has some correlation in space,
so that the equation admits a solution as a real-valued process
(or random field); see the works of \cite{DF98, D99, CD08}.
In particular, their results show that when $\dot{W}$ is white in time, i.e., $H = 1/2$, 
the equation \eqref{SWE} has a unique random field solution if and only if $N-\beta < 2$.
The fractional case $1/2 < H < 1$ was studied in \cite{BT10}.
It was proved that \eqref{SWE} has a unique random field solution if and only if 
$N-\beta < 2H+1$. 
Besides existence and uniqueness, 
results on space-time H\"older regularities and hitting probabilities for the solution of 
\eqref{SWE} can be found in \cite{DS10, CT14}.
To the author's knowledge, not much else is known about 
\eqref{SWE} in the case of $H \ne 1/2$ and $N \ge 2$.

The purpose of this paper is to study the local times (or occupation densities) 
of the stochastic wave equation driven by fractional-colored noise.
The solution of \eqref{SWE} is an $\R^d$-valued Gaussian random field
$u = \{ u(t, x) : t \ge 0, x \in \R^N\}$.
Our approach to the study of local times of $u$ 
is based on the Fourier analytic method due to Berman \cite{B69, B73} 
and the use of the local nondeterminism property, 
which is one of the main tools for studying local times of Gaussian random fields.
This approach was used in \cite{OT13} to study local times of the stochastic heat 
equation in the time variable; a different approach was used in \cite{T13}.

The notion of local nondeterminism (LND) for Gaussian processes 
was first introduced by Berman \cite{B73} to study the existence of jointly continuous
local times, and was extended by Pitt \cite{P78} to study the local 
times of multivariate Gaussian random fields.
Various forms of LND have been studied in the literature, e.g., 
\cite{C78, CD82, MP87, KWX06, X09}.
In particular, the property of strong local nondeterminism has been developed.
An example of Gaussian random field that satisfies this stronger property of LND
is the multiparameter fractional Brownian motion \cite{P78}.
The investigation of the strong LND property is of interest because 
it has found applications in studying 
various properties of Gaussian random fields such as exact modulus of continuity of 
sample paths, small ball probabilities and fractal properties. 
We refer to the survey of Xiao \cite{X08} for details.

A result of Walsh \cite[Theorem 3.1]{W86} shows that the solution of 
the linear stochastic wave equation in one spatial dimension driven by 
the space-time white noise can be represented by a modified Brownian 
sheet.
It is known that the Brownian sheet does not satisfy the LND property in the sense of 
Berman or Pitt, but it satisfies a different type of LND called sectorial LND 
\cite{KWX06, KX04}.
This leads to the natural question of whether the solution of the stochastic wave 
equation satisfies the LND property.

Recently, Lee and Xiao \cite{LX19} considered the linear stochastic wave equation 
driven by a Gaussian noise that is white in time and colored in space with 
spatial covariance given by the Riesz kernel, and showed that 
the solution satisfies a new type of strong LND in the form of a spherical integral,
which turns out to be useful in proving the exact modulus of continuity of the sample 
functions.

The contributions of the present paper are as follows.
We extend the result of \cite{LX19} to the case of fractional-colored noise
and prove that the solution $u$ of \eqref{SWE} satisfies a spherical integral form of strong LND in $(t, x)$ (Proposition \ref{prop:LND}). 
We also study the LND property of $u$ in $t$ or $x$ when the other variable is 
held fixed. Since the LND property in joint variable $(t, x)$ takes a different form from 
those that are studied previously, e.g., in \cite{B73, P78, GH80, X97, WX11}, 
their local times results cannot be applied directly.
In this case, we exploit this new type of LND property to study the local 
times of the stochastic wave equation.
More specifically, in Theorem \ref{thm:JC}, we prove that $u$, 
as a Gaussian random field $(t, x) \mapsto u(t, x)$, 
has a jointly continuous local time if $\frac{1}{2}(2H+1-N+\beta)d < 1+N$.
Moreover, we obtain differentiability results for the local times in the space variable,
and the local and uniform moduli of continuity of the local times in the set variable (Theorems \ref{thm:D} and \ref{MC:L*}).
Our results lead to sample function properties that are new for the stochastic wave 
equation with fractional-colored noise, 
including the exact uniform modulus of continuity of $(t, x) \mapsto u(t, x)$ 
(Theorem \ref{thm:MC})
and the property that $(t, x) \mapsto u(t, x)$ is nowhere differentiable 
(Theorem \ref{thm:ND}).

The rest of the paper is organized as follows.
In Section 2, we recall the theory of the stochastic wave equation with 
fractional-colored noise and some facts about local times.
In Section 3, we study the local nondeterminism property for the solution $u$ 
of \eqref{SWE} and give the exact modulus of continuity of the sample functions.
In Section 4, we derive moment estimates for the local times of $u$ using the LND 
property and prove the existence of jointly continuous local times.
Finally, in Section 5, we study regularities of the local times
including the differentiability in the set variable and the moduli of continuity in the set 
variable, and we use the latter to derive lower envelops of the sample path oscillations.

\section{Preliminaries}

This section contains preliminaries about the stochastic wave equation and local times.
We first introduce some notations and review the theory for
fractional-colored noises and the solution of the equation \eqref{SWE}.
After that, we will recall the definition and some properties of the local times 
of stochastic processes.

The Fourier transform of a function $\phi: \R^N \to \R$ (or $\mathbb{C}$) is defined by 
\[\mathscr{F}\phi(\xi) = \widehat\phi(\xi) = \int_{\R^N} e^{-ix \cdot \xi}\phi(x) dx\]
whenever the integral is well-defined.
We write $x \cdot \xi = \sum_{j = 1}^N x_j \xi_j$ for the usual dot product in $\R^N$.
We assume that, the function $f_\beta$ in \eqref{Wcov}, where $0 < \beta < N$, 
is the Fourier transform of a measure $\mu_\beta$ in the sense that 
\[ \int_{\R^N} f_\beta(x) \phi(x) dx = \int_{\R^N} \mathscr{F}\phi(\xi) \mu_\beta(d\xi) \]
for all $\phi$ in the Schwartz space $\mathcal{S}(\R^N)$ of rapidly decreasing functions, 
and that $\mu_\beta$ has a density $h_\beta(\xi)$ with respect to the Lebesgue measure satisfying
\begin{equation}\label{h}
C_1|\xi|^{-\beta} \le h_\beta(\xi) \le C_2 |\xi|^{-\beta},
\end{equation}
where $C_1$ and $C_2$ are positive finite constants.
A typical example of $f_\beta$ is the Riesz kernel:
\[ f_\beta(\xi) = |\xi|^{-(N-\beta)}, \quad 0 < \beta < N, \]
which is the Fourier transform of $\mu_\beta(d\xi) = C |\xi|^{-\beta} d\xi$,
where $C$ is a suitable constant depending on $\beta$ and $N$.
See \cite[\S V]{S}.

Let us recall the random field approach for the fractional-colored noise $\dot W$
and the construction of the solution of \eqref{SWE} proposed by Balan and Tudor 
\cite{BT08, BT10}.
Let $C_c^\infty(\R_+ \times \R^N)$ denote the space of smooth, compactly supported
functions on $\R_+ \times \R^N$,
and define the inner product $\langle \cdot, \cdot \rangle_{\mathcal{HP}}$ on 
$C_c^\infty(\R_+ \times \R^N)$ by
\begin{equation}\label{HP}
\begin{split}
\langle \varphi, \psi\rangle_{\mathcal{HP}}
& = \int_{\R_+} dt \int_{\R_+} ds \int_{\R^N} dx \int_{\R^N} dy\,
\varphi(t, x) \, \rho_H(t-s) f_\beta(x-y)\, \psi(s, y)\\
& = C \int_{\R_+} dt \int_{\R_+} ds \int_{\R^N} d\xi\,
\mathscr{F}\varphi(t, \cdot)(\xi)\, \overline{\mathscr{F}\psi(s, \cdot)(\xi)} \,
\rho_H(t-s) \, h_\beta(\xi).
\end{split}
\end{equation}
The Hilbert space $\mathcal{HP}$ associated with the noise $\dot W$
is defined as the completion of 
$C_c^\infty(\R_+ \times \R^N)$ with respect to the inner product 
$\langle \cdot, \cdot \rangle_{\mathcal{HP}}$.
This Hilbert space can be identified as the space of all distribution-valued functions
$S: \R_+ \to \mathcal{S}'(\R^N)$ such that for each $t \ge 0$, 
$\mathscr{F}S(t)$ is a function, and
\[ \int_{\R_+} dt \int_{\R_+} ds \int_{\R^N} d\xi\, 
\mathscr{F}S(t)(\xi)\, \overline{\mathscr{F}S(s)(\xi)} 
\, \rho_H(t-s) \, h_\beta(\xi) < \infty. \]
The Gaussian noise $\dot W$ can be defined as a generalized 
Gaussian process $\{ W(\varphi) : \varphi \in C^\infty_c(\R_+\times \R^N)\}$ 
with mean zero and covariance
$\E[W(\varphi)W(\psi)] = \langle \varphi, \psi\rangle_{\mathcal{HP}}$.
Then $W$ produces an isometry from $\mathcal{HP}$ into 
a Gaussian subspace of $L^2(\P)$
\[ \varphi \mapsto W(\varphi) =: \int_{\R_+} \int_{\R^N} \varphi(s, y) W(ds, dy), \]
such that $\E[W(\varphi)W(\psi)] = \langle \varphi, \psi\rangle_{\mathcal{HP}}$
for all $\varphi, \psi \in \mathcal{HP}$.

Let $g_{t, x}(s, y) = G(t-s, x-y) {\bf 1}_{[0, t]}(s)$, where $G(t, x)$ is the fundamental 
solution of the wave equation.
By Theorem 3.1 of \cite{BT10}, under the assumption \eqref{h},
$g_{t, x} \in \mathcal{HP}$ if and only if 
\begin{equation}\label{beta0}
\beta > N - 2H - 1,
\end{equation}
and when \eqref{beta0} holds, 
\eqref{SWE} has a random field solution that is mean square continuous in $(t, x)$:
\[ u(t, x) = W(g_{t, x}) = \int_0^t \int_{\R^N} G(t-s, x-y) W(ds, dy). \]
Note that $u = \{ u(t, x) : t \ge 0, x \in \R^N \}$ is an $\R^d$-valued 
Gaussian field with i.i.d.\ components.

If the Fourier transforms of $s \mapsto \mathscr{F}\varphi(s, \cdot)(\xi)$ and 
$s \mapsto \mathscr{F}\psi(s, \cdot)(\xi)$ exist,
they are equal to the Fourier transforms of $\varphi$ and $\psi$ in $(t, x)$-variables,
denoted by $\mathscr{F}\varphi(\tau, \xi)$ and $\mathscr{F}\psi(\tau, \xi)$, 
respectively.
In this case, it follows from \eqref{HP} and the Plancherel theorem that 
\begin{equation}\label{HPF}
\langle \varphi, \psi\rangle_{\mathcal{HP}}
= C \int_\R d\tau \int_{\R^N} d\xi\, 
\mathscr{F}\varphi(\tau, \xi)\, \overline{\mathscr{F}\psi(\tau, \xi)}\, |\tau|^{1-2H} 
\,h_\beta(\xi).
\end{equation}
This formula remains true when $H = 1/2$.
Recall that the Fourier transform of the fundamental solution in the space variable
is $\mathscr{F}G(t, \cdot)(\xi) = |\xi|^{-1}\sin(t|\xi|)$; see \cite[Ch.5]{F}. 
Then
\begin{equation*}
\mathscr{F}g_{t, x}(s, \cdot)(\xi) = e^{-ix\cdot \xi}\, \frac{\sin((t-s)|\xi|)}{|\xi|} {\bf 1}_{[0, t]}(s)
\end{equation*}
and the Fourier transform of $s \mapsto \mathscr{F}g_{t, x}(s, \cdot)(\xi)$ is 
\begin{equation}\label{Fg}
\mathscr{F}g_{t, x}(\tau, \xi) = \frac{e^{-ix\cdot \xi}}{2|\xi|}
\left( \frac{e^{-it\tau} - e^{it|\xi|}}{\tau+|\xi|} - \frac{e^{-it\tau}-e^{-it|\xi|}}{\tau - |\xi|} \right).
\end{equation}
This and \eqref{HPF} provide a formula for the covariance of $u(t, x)$.

Next, let us recall the definition and properties of local times.
Let $u = \{u(z) : z \in \R^k \}$ be a random field with values in $\R^d$.
The occupation measure of $u$ on a Borel set $T \in \mathscr{B}(\R^k)$ is 
the random measure defined by
\[
\nu_T(B) = \nu_T(B, \omega) 
= \lambda_k\{ z \in T : u(z) \in B \}, \quad B \in \mathscr{B}(\R^d),
\]
where $\lambda_k$ denotes the Lebesgue measure on $\R^k$.
We say that $u$ has a local time on $T$ if $\nu_T$ is a.s.\ absolutely continuous 
with respect to $\lambda_d$, the Lebesgue measure on $\R^d$.
A version of the Radon--Nikodym derivative, denoted by $L(v, T) = L(v, T, \omega)$, 
$v \in \R^d$, is called a version of the local time.
It follows from the definition that for all $B \in \mathscr{B}(\R^d)$,
\begin{equation}\label{Def:LT}
\lambda_k\{ z \in T : u(z) \in B \} = \int_B L(v, T) \, dv.
\end{equation}
Obviously, if $u$ has a local time on $T$, then it has a local time on 
every $S$ in $\mathscr{B}(T)$, the collection of all Borel subsets of $T$.
We say that $L(v, S)$ is a kernel if
\begin{enumerate}
\item[(i)] For each $S \in \mathscr{B}(T)$, the function 
$(v, \omega) \mapsto L(v, S, \omega)$ is 
$\mathscr{B}(\R^d) \times \mathscr{F}$-measurable;
\item[(ii)] For each $(v, \omega) \in \R^d \times \Omega$, the set function 
$S \mapsto L(v, S, \omega)$ is a measure on $(T, \mathscr{B}(T))$.
\end{enumerate}
It is desirable to work with a version of the local time that is a kernel because 
it satisfies the following properties \cite[Theorem (6.4)]{GH80}:
\begin{enumerate}
\item[(i)] Occupation density formula: for every nonnegative Borel function $f(z, v)$ on $T \times \R^d$,
\begin{equation*}
\int_T f(z, u(z)) \, dz = \int_{\R^d} dv \int_T f(z, v) L(v, dz).
\end{equation*}
\item[(ii)] $L(v, M_v^c) = 0$ for a.e.~$v$, where $M_v = \{ z \in T : u(z) = v \}$, i.e., 
the support of the measure $L(v, \cdot)$ is contained in the $v$-level set $M_v$ of $u$.
\end{enumerate}

For any $z = (z_1, \dots, z_k) \in \R^k$, let $(-\infty, z]$ denote the unbounded interval
$\prod_{j=1}^k (-\infty, z_j]$ in $\R^k$ with upper right corner at $z$.
Let $T$ be a compact interval in $\R^k$ and $Q_z = (-\infty, z] \cap T$.
We say that a version of the local time $L$ is jointly continuous on $T$ if 
$(v, z) \mapsto L(v, Q_z)$ is jointly continuous on $\R^d \times T$.
It is known that 
if $L$ is jointly continuous, then $L(v, Q_z)$ can be uniquely extended to a kernel 
$L(v, S)$, $S \in \mathscr{B}(T)$, satisfying the property that $L(v, J) = 0$ 
for all $v \not\in \overline{u(J)}$ and all intervals $J \subset T$ with rational endpoints;
and if, in addition, $u$ has continuous sample paths, then 
$L(v, M_v^c) = 0$ for all $v \in \R^d$;
see \cite{B70}, \cite[p.12]{GH80}, \cite[p.223]{A}.
Since the local times serve as a natural measure on the level sets of $u$, 
they are useful in studying properties of the level sets
\cite{B72, GH80, MP87, X97}.

\medskip

\section{Local nondeterminism}

This section is devoted to studying LND property
of the solution $u(t, x)$ of \eqref{SWE}.
In what follows, we denote
\begin{equation*}
\alpha = \frac{2H+1-N+\beta}{2}
\end{equation*}
and assume
\begin{equation}\label{beta}
N - 2H - 1 < \beta < N - 2H + 1
\end{equation}
so that the solution $u(t, x)$ exists by \eqref{beta0}, and that $0 < \alpha < 1$. 
In this case, the proposition below implies that the sample functions 
$(t, x) \mapsto u(t, x)$ are a.s.\ locally H\"older continuous of any order 
strictly less than $\alpha$.

\begin{proposition}\label{prop:HR}
For any $0 < a < b$ and $M > 0$, there exist positive finite constants 
$C_1$ and $C_2$ such that for all $(t, x), (s, y) \in [a, b] \times [-M, M]^N$,
\begin{equation}\label{SWEvar}
C_1 (|t-s| + |x-y|)^{2\alpha} \le \E(|u(t, x) - u(s, y)|^2) \le 
C_2 (|t-s| + |x-y|)^{2\alpha}.
\end{equation}
\end{proposition}

\begin{proof}
We may assume that $d = 1$.
By \eqref{HPF}, we have
\begin{equation*}
\E(|u(t, x) - u(s, y)|^2)
= C \int_\R d\tau \int_{\R^N} d\xi\,
|\mathscr{F} g_{t, x}(\tau, \xi) - \mathscr{F} g_{s, y}(\tau, \xi)|^2\,
|\tau|^{1-2H}\, h_\beta(\xi).
\end{equation*}
Then, by the assumption \eqref{h} for $h_\beta$, 
it is enough to consider the case that $h_\beta(\xi) = |\xi|^{-\beta}$.
But for this case, \eqref{SWEvar} has been proved in \cite{CT14}.
\end{proof}

The following LND result is the basis for the study of regularity properties of local times 
in this paper.
This result extends Proposition 2.1 of \cite{LX19} and shows that 
$u(t, x)$ satisfies a strong LND property in the form of a spherical integral.
The proof involves a Fourier analytic method.

\begin{proposition}\label{prop:LND}
For any $0 < a < \infty$, there exist constants $C > 0$ and $r_0 > 0$ 
such that for all integers $n \ge 1$, for all 
$(t, x), (t^1, x^1), \dots, (t^n, x^n) \in [a, \infty) \times \R^N$ with 
$\max_j(|t-t^j| + |x-x^j|) \le r_0$, we have
\begin{equation}\label{LND:ineq}
\mathrm{Var}(u_1(t, x)|u_1(t^1, x^1), \dots, u_1(t^n, x^n)) 
\ge C \int_{\mathbb{S}^{N-1}} \min_{1 \le j \le n}|(t-t^j) + (x-x^j) \cdot w|^{2\alpha} 
\,\sigma(dw),
\end{equation}
where $\sigma$ is the surface measure on the unit sphere $\mathbb{S}^{N-1}$.
\end{proposition}

\begin{proof}
Take $r_0 = a/2$. For each $w \in \mathbb{S}^{N-1}$, let
\[ r(w) = \min_{1 \le j \le n} |(t-t^j) + (x - x^j) \cdot w|. \]
Since $u$ is Gaussian, the conditional variance
$\mathrm{Var}(u_1(t, x)|u_1(t^1, x^1), \dots, u_1(t^n, x^n))$ is the squared distance 
between $u(t, x)$ and the linear subspace spanned by 
$u_1(t^1, x^1), \dots, u_1(t^n, x^n)$ in $L^2(\P)$.
Thus, it suffices to prove that there exist constants $C> 0$ and $r_0 > 0$ such that 
for any $n \ge 1$,
for any $(t, x), (t^1, x^1), \dots, (t^n, x^n) \in [a, \infty) \times \R^N$
with $\max_j(|t-t^j| + |x-x^j|) \le r_0$,
and for any choice of real numbers $a_1, \dots, a_n$, we have
\begin{equation}\label{Eq:LND}
\E\bigg[\Big(u_1(t, x) - \sum_{j=1}^n a_j u_1(t^j, x^j)\Big)^2\bigg]
\ge C \int_{\mathbb{S}^{N-1}} r(w)^{2H+1-N+\beta} \,\sigma(dw).
\end{equation}
To this end, we note that by \eqref{h} and \eqref{HPF},
\begin{equation*}
\begin{split}
&\E\bigg[\Big(u_1(t, x) - \sum_{j=1}^n a_j u_1(t^j, x^j)\Big)^2\bigg]\\
& \quad \ge C \int_\R d\tau \int_{\R^N} d\xi \,\Big|\mathscr{F}g_{t, x}(\tau, \xi) 
- \sum_{j=1}^n a_j \mathscr{F}g_{t^j, x^j}(\tau, \xi)\Big|^2 
|\tau|^{1-2H} |\xi|^{-\beta}.
\end{split}
\end{equation*}
Then, we use \eqref{Fg} and spherical coordinates $\xi = \rho w$ to get
\begin{align*}
C \int_\R d\tau \int_{\R_+} d\rho\,\int_{\mathbb{S}^{N-1}} \sigma(dw) 
\Big| F(t, x \cdot w, \tau, \rho) - 
\sum_{j=1}^n a_j  F(t^j, x^j \cdot w, \tau, \rho)\Big|^2 
|\tau|^{1-2H} \rho^{N-\beta - 3},
\end{align*}
where
\[  F(t, y, \tau, \rho) 
= \frac{e^{-i \rho y}}{2}
\left( \frac{e^{-it\tau} - e^{it\rho}}{\tau+\rho} - \frac{e^{-it\tau}-e^{-it\rho}}{\tau - \rho} \right). 
\]
Since $ F(t, x \cdot w, -\tau, -\rho) =- \overline{ F(t, x \cdot w, \tau, \rho)}$,
it follows that
\begin{equation}
\begin{split}\label{E2}
&\E\bigg[\Big(u_1(t, x) - \sum_{j=1}^n a_j u_1(t^j, x^j)\Big)^2\bigg]\\
& \ge \frac{C}{2} \int_{\mathbb{S}^{N-1}} \sigma(dw) \underbrace{\int_\R d\tau \int_\R d\rho\,
\Big| F(t, x \cdot w, \tau, \rho) - 
\sum_{j=1}^n a_j  F(t^j, x^j \cdot w, \tau, \rho)\Big|^2 
|\tau|^{1-2H} |\rho|^{N-\beta-3}}_{=:A(w)}.
\end{split}
\end{equation}

Choose and fix any two nonnegative smooth test functions $\phi$, $\psi: \R \to \R$ 
satisfying the following properties:
$\phi$ is supported on $[0, a/2]$ and $\int \phi(s) ds = 1$;
$\psi$ is supported on $[-1, 1]$ and $\psi(0) = 1$.
Let $\psi_r(x) = r^{-1} \psi(r^{-1} x)$.
For each $w \in \mathbb{S}^{N-1}$ with $r(w) > 0$, 
write $\widehat{\psi}_{r(w)} = \mathscr{F}(\psi_{r(w)})$ and consider
\[
I(w) := \int_\R d\rho \int_\R d\tau \, \overline{\bigg( F(t, x \cdot w, \tau, \rho) - 
\sum_{j=1}^n a_j F(t^j, x^j \cdot w, \tau, \rho)\bigg)}\,
e^{-it\rho}e^{-i\rho x\cdot w}\widehat\phi(\tau-\rho) \widehat\psi_{r(w)}(\rho).
\]
Note that for $\rho$ fixed,
$\tau \mapsto \widehat\phi(\tau-\rho)$ is the Fourier transform of
$s \mapsto e^{is\rho} \phi(s)$, 
and $\tau \mapsto F(t, x\cdot w, \tau, \rho)$ is the Fourier transform of
$s \mapsto e^{-i\rho x\cdot w}\sin((t-s)\rho) {\bf 1}_{[0, t]}(s)$.
Then apply the Plancherel theorem to the integral in $\tau$ to get that
\begin{align*}
I(w) 
&= 2\pi \int_\R d\rho \int_\R ds \, \overline{\bigg( e^{-i\rho x\cdot w}\sin((t-s)\rho) 
{\bf 1}_{[0, t]}(s) - 
\sum_{j=1}^n a_j e^{-i\rho x^j \cdot w} \sin((t^j-s)\rho) {\bf 1}_{[0, t^j]}(s) \bigg)}\\
& \hspace{255pt}\times e^{-i(t-s)\rho} e^{-i\rho x\cdot w}\phi(s) \widehat\psi_{r(w)}(\rho).
\end{align*}
Since $\sin(z) =\frac{1}{2i} (e^{iz} - e^{-iz})$ and $\phi$ is supported on $[0, a/2]$,
this is
\begin{align*}
&= -\pi i
\int_0^{a/2} ds \int_\R d\rho\, 
\bigg[ e^{i\rho x\cdot w}\Big(e^{i(t-s)\rho} - e^{-i(t-s)\rho}\Big) \\
& \hspace{80pt} - 
\sum_{j=1}^n a_j e^{i\rho x^j \cdot w} \Big(e^{i(t^j-s)\rho} - e^{-i(t^j-s)\rho}\Big)\bigg]
e^{-i(t-s)\rho} e^{-i\rho x\cdot w}\phi(s) \widehat\psi_{r(w)}(\rho).
\end{align*}
Then, apply the Fourier inversion theorem to 
$\widehat\psi_{r(w)} = \mathscr{F}(\psi_{r(w)})$ to get
\begin{align*}
&= -2\pi^2 i
\int_0^{a/2} \phi(s) \bigg[ \psi_{r(w)}(0) - \psi_{r(w)}(-2(t-s)) \\
& \qquad - \sum_{j=1}^n a_j \Big(\psi_{r(w)}((x^j-x)\cdot w + (t^j-t)) - 
\psi_{r(w)}((x^j-x)\cdot w + (t^j-t) - 2(t^j-s))\Big)\bigg] ds.
\end{align*}
Since $t \ge a$ and $r(w) \le r_0 = a/2$, we see that 
for all $s \in [0, a/2]$, $2(t-s)/r(w) \ge 2$ and thus
\[ \psi_{r(w)}(-2(t-s)) = 0. \]
By the definition of $r(w)$, we have $|(x^j-x)\cdot w + (t^j - t)|/r(w) \ge 1$, which implies
\[ \psi_{r(w)}((x^j-x)\cdot w + (t^j-t)) = 0. \]
Moreover, we have $(x^j-x)\cdot w + (t^j-t) - 2(t^j-s) \le r_0 - a = -a/2 \le -r(w)$, 
hence
\[ \psi_{r(w)}((x^j-x)\cdot w + (t^j-t) - 2(t^j-s)) = 0. \]
It follows that
\begin{equation}\label{I}
|I(w)| = 2\pi^2 \psi_{r(w)}(0) \int_0^{a/2} \phi(s)\, ds = 2\pi^2 r(w)^{-1}.
\end{equation}

On the other hand, by the Cauchy--Schwarz inequality,
\begin{align}\label{LND:int}
|I(w)|^2 \le A(w) \times \int_\R d\tau \int_\R d\rho \,
|\widehat\phi(\tau-\rho)|^2  |\widehat\psi_{r(w)}(\rho)|^2 |\tau|^{2H-1} |\rho|^{3-N+\beta}.
\end{align}
Note that $\widehat\psi_{r(w)}(\rho) = \widehat\psi(r(w)\rho)$ and
both $\widehat \phi$ and $\widehat \psi$ are rapidly decreasing functions.
To estimate the double integral in \eqref{LND:int}, 
we consider two regions: (i) $|\tau|\le |\rho|$ and 
(ii) $|\tau| > |\rho|$.
For region (i), by $|\tau|^{2H-1} \le |\rho|^{2H-1}$ and 
scaling in $\rho$, we have
\begin{align*}
&\int_\R d\rho \int_{|\tau| \le |\rho|} d\tau \, 
|\widehat\phi(\tau-\rho)|^2  |\widehat\psi_{r(w)}(\rho)|^2 |\tau|^{2H-1} |\rho|^{3-N+\beta}\\
& \le \int_\R d\rho\, |\widehat\psi(r(w)\rho)|^2 |\rho|^{2H+2-N+\beta} 
\int_\R d\tau \,|\widehat\phi(\tau)|^2\\
& = C r(w)^{-2H-3+N-\beta}.
\end{align*}
For region (ii), note that $3-N+\beta > 2H+1-N+\beta > 0$ by \eqref{beta},
so $|\rho|^{3-N+\beta} \le |\tau|^{3-N+\beta}$.
By letting $z = \tau - \rho$ and then by scaling,
\begin{align*}
&\int_\R d\rho \int_{|\tau| > |\rho|} d\tau \, 
|\widehat\phi(\tau-\rho)|^2  |\widehat\psi_{r(w)}(\rho)|^2 |\tau|^{2H-1} |\rho|^{3-N+\beta}\\
& \le \int_\R d\rho\, |\widehat\psi(r(w)\rho)|^2 
\int_\R dz \,|\widehat\phi(z)|^2 |z + \rho|^{2H+2-N+\beta} \\
& \le C \int_\R d\rho \, |\widehat\psi(r(w)\rho)|^2 
\int_\R dz\, |\widehat\phi(z)|^2 |z|^{2H+2-N+\beta} 
+ C \int_\R d\rho \, |\widehat\psi(r(w)\rho)|^2  |\rho|^{2H+2-N+\beta}
\int_\R dz\, |\widehat\phi(z)|^2\\
& \le C r(w)^{-1} + C r(w)^{-2H-3+N-\beta},
\end{align*}
which is $\le C r(w)^{-2H-3+N-\beta}$ for some larger constant $C$ 
because $2H+3-N+\beta > 1$.
Hence
\begin{equation}\label{I^2}
|I(w)|^2 \le C A(w) r(w)^{-2H-3+N-\beta}.
\end{equation}

Now, combining \eqref{I} and \eqref{I^2}, we get that 
\begin{equation}\label{A}
A(w) \ge C r(w)^{2H+1-N+\beta},
\end{equation}
and this remains true if $r(w) = 0$.
Therefore, we can integrate both sides of \eqref{A} over $\mathbb{S}^{N-1}$ 
with respect to $\sigma(dw)$ and use \eqref{E2} to get \eqref{Eq:LND}.
\end{proof}

When $N = 1$, $\sigma$ is supported on $\{-1, 1\}$.
In this case, $u(t, x)$ satisfies sectorial LND
under the change of coordinates $(t, x) \mapsto (t+x, t-x)$.
It is known that the Brownian sheet and fractional Brownian sheets 
satisfy the sectorial LND property \cite{KX04, WX07}.

\begin{corollary}\label{cor:sectLND}
When $N = 1$, \eqref{LND:ineq} becomes
\begin{equation*}
\begin{split}
&\mathrm{Var}(u_1(t, x)|u_1(t^1, x^1), \dots, u_1(t^n, x^n)) \\
& \qquad \ge C \Big( \min_{1 \le j \le n}|(t+x) - (t^j+x^j)|^{2\alpha} 
+ \min_{1\le j \le n} |(t-x) - (t^j-x^j)|^{2\alpha} \Big).
\end{split}
\end{equation*}
\end{corollary}

Moreover, $u(t, x)$ satisfies the strong LND property in one variable ($t$ or $x$)
while the other variable is held fixed.

\begin{corollary}\label{cor:LND}
Fix $x_0 \in \R^N$.
For any $0 < a < \infty$, there exist constants $C > 0$ and $r_0 > 0$ 
such that for all integers $n \ge 1$, for all $t, t^1, \dots, t^n \in [a, \infty)$ with 
$\max_j |t-t^j| \le r_0$, we have
\begin{equation*}
\mathrm{Var}(u_1(t, x_0) | u_1(t^1, x_0), \dots, u_1(t^n, x_0))
\ge C \min_{1 \le j \le n} |t-t^j|^{2\alpha}.
\end{equation*}
\end{corollary}

\begin{proposition}\label{prop:LNDx}
Fix $t_0 > 0$. Then $\{u_1(t_0, x) : x \in \R^N \}$ is a stationary Gaussian random field
with a spectral density $f_{t_0}(\xi) = N_{t_0}(\xi) h_\beta(\xi)$, $\xi \in \R^N$, where
\[ N_{t_0}(\xi) = C |\xi|^{-2} \int_\R \bigg| \frac{e^{-it_0 \tau} - e^{it_0|\xi|}}{\tau+|\xi|} - \frac{e^{-it_0\tau} - e^{-it_0|\xi|}}{\tau-|\xi|} \bigg|^2 |\tau|^{1-2H} \, d\tau.\]
Moreover, for any $0 < M < \infty$, there exists a constant $C > 0$ such that 
for all integers $n \ge 1$, for all $x, x^1, \dots, x^n \in \{ y \in \R^N : |y| \le M \}$,
\begin{equation}\label{LNDx}
\mathrm{Var}(u_1(t_0, x) | u_1(t_0, x^1), \dots, u_1(t_0, x^n)) 
\ge C \min_{1 \le j \le n} |x - x^j|^{2\alpha}.
\end{equation}
\end{proposition}

\begin{proof}
By \eqref{HPF} and \eqref{Fg}, we have $\E[u_1(t_0, x) u_1(t_0, y)]
= \int_{\R^n} e^{-i (x - y) \cdot \xi} N_{t_0}(\xi) h_\beta(\xi)\, d\xi$, 
which verifies the first assertion. 
To prove \eqref{LNDx}, use \eqref{HPF}, \eqref{h} and \eqref{Fg} to get that, 
for any $a_1, \dots, a_n \in \R$,
\begin{align*}
\E\bigg[ \Big( u_1(t_0, x) - \sum_{j=1}^n a_j u_1(t_0, x^j) \Big)^2 \bigg]
\ge C \int_{\R^N} \Big| 1 - \sum_{j=1}^n a_j e^{i(x-x^j)\cdot \xi} \Big|^2 
N_{t_0}(\xi) |\xi|^{-\beta} \, d\xi.
\end{align*}
By Lemma 6.2 of \cite{B12}, there exists a constant $C>0$ depending on $t_0$ such that for all $\xi \in \R^N$,
\[ N_{t_0}(\xi) \ge \frac{C}{\sqrt{|\xi|^2 + 1}} \int_\R \frac{|\tau|^{1-2H}}{|\tau|^2 + |\xi|^2 + 1} d\tau. \]
Fix any two nonnegative smooth test functions $\phi: \R \to \R$ and $\psi: \R^N \to \R$
satisfying the following properties: $\phi$ is supported on $[-1, 1]$, 
$\psi$ is supported on $\{ \xi \in \R^N : |\xi| \le 1 \}$, and $\phi(0) = \psi(0) = 1$.
Let $r = \min_j |x - x^j|$, $\phi_r(\tau) = r^{-1}\phi(r^{-1}\tau)$, 
$\psi_r(\xi) = r^{-N}\psi(r^{-1}\xi)$ and consider
\[ I := \iint_{\R \times \R^N} \Big( 1 - \sum_{j=1}^n a_j e^{i(x-x^j)\cdot \xi} \Big) \widehat{\phi}_r(\tau) \widehat{\psi}_r(\xi)\, d\tau \, d\xi. \]
By Fourier inversion, 
\begin{equation}\label{LNDx:I}
I = (2\pi)^{1+N} \phi_r(0) \Big(\psi_r(0) - \sum_{j=1}^n a_j \psi_r(x-x^j)\Big)
= (2\pi)^{1+N} r^{-1-N}.
\end{equation}
On the other hand, by the Cauchy--Schwarz inequality,
\begin{align*}
I^2 & \le C\, \E\bigg[ \Big( u_1(t_0, x) - \sum_{j=1}^n a_j u_1(t_0, x^j) \Big)^2 \bigg]\\
& \quad \times \iint_{\R \times \R^N} \sqrt{|\xi|^2 + 1}\, (|\tau|^2 + |\xi|^2 + 1) 
|\tau|^{2H-1} |\xi|^\beta |\widehat{\phi}(r\tau) \widehat{\psi}(r\xi)|^2\, d\tau\, d\xi.
\end{align*}
By scaling, the double integral is equal to
\begin{align*}
r^{-2H-3-N-\beta} \iint_{\R \times \R^N} \sqrt{|\xi|^2 + r^2}\, (|\tau|^2 + |\xi|^2 + r^2)
|\tau|^{2H-1} |\xi|^\beta |\widehat{\phi}(\tau) \widehat{\psi}(\xi)|^2 \, d\tau\, d\xi,
\end{align*}
which is $\le C r^{-2H-3-N-\beta}$ by applying $r \le 2M$ to the integrand. 
This and \eqref{LNDx:I} imply that
\[\E\bigg[ \Big( u_1(t_0, x) - \sum_{j=1}^n a_j u_1(t_0, x^j) \Big)^2 \bigg]
\ge C r^{2H+1-N+\beta}, \]
where $C$ does not depend on $n, x, x^j$ or $a_j$.
This proves \eqref{LNDx}.
\end{proof}

A property of the conditional variances 
$\mathrm{Var}(u_1(t, x)| u_1(t^1, x^1), \dots, u_1(t^n, x^n))$ 
is that they are strictly positive
whenever the points $(t^j, x^j)$ are all different from $(t, x)$.
Indeed, $u$ has the following linear independence property:

\begin{proposition}\label{lem:LI}
For any $n \ge 2$, for any distinct points
$(t^1, x^1), \dots, (t^n, x^n)$ in $(0, \infty) \times \R^N$, 
the Gaussian random variables 
$u_1(t^1, x^1), \dots, u_1(t^n, x^n)$ are linearly independent.
\end{proposition}

\begin{proof}
Suppose $a_1, \dots, a_n$ are real numbers such that
$\sum_{j=1}^n a_j u_1(t^j, x^j) = 0$ a.s. Then by \eqref{HPF},
\[
0 = \E\bigg(\sum_{j=1}^n a_j u_1(t^j, x^j) \bigg)^2
= \int_\R d\tau \int_{\R^N} d\xi \,\bigg| \sum_{j=1}^n
a_j \mathscr{F} g_{t^j, x^j}(\tau, \xi) \bigg|^2 |\tau|^{1-2H} h_\beta(\xi).
\]
It follows that for all $\tau \in \mathbb{R}$ and $\xi \in \mathbb{R}^N$,
$\sum_{j=1}^n a_j \mathscr{F}g_{t^j, x^j}(\tau, \xi) = 0$, which, 
by \eqref{Fg}, implies
\begin{equation}\label{LI:eq1}
\sum_{j=1}^n b_j e^{-i t^j \tau} + c_1 \tau + c_2 = 0,
\end{equation}
where $b_j = -2 a_j |\xi| e^{-i x^j \cdot \xi}$,
\begin{align*}
c_1 &= -\sum_{j=1}^n a_j e^{-i x^j \cdot \xi}(e^{it^j|\xi|} - e^{-it^j|\xi|}),\\
c_2 &= \sum_{j=1}^n a_j |\xi| e^{-i x^j \cdot \xi}(e^{it^j|\xi|} + e^{-it^j|\xi|}).
\end{align*}
We need to show that $a_j = 0$ for all $j = 1, \dots, n$.
Let $\hat{t}^1, \dots, \hat{t}^p$ be all distinct values of the $t^j$'s.
If we fix an arbitrary $\xi \in \mathbb{R}^N$ and differentiate \eqref{LI:eq1}
with respect to $\tau$, we see that for all $\tau \in \R$,
\[
\sum_{\ell = 1}^p \bigg(-i\hat{t}^\ell \sum_{j: t^j = \hat{t}^\ell} b_j\bigg)
e^{-i \hat{t}^\ell \tau} + c_1 = 0.
\]
Since the functions $\{e^{-i \hat{t}^1 \tau}, \dots,
e^{-i \hat{t}^p \tau}, 1 \}$ are linearly independent over $\mathbb{C}$, we have
\[ -i\hat{t}^\ell \sum_{j: t^j = \hat{t}^\ell} b_j= 0 \]
for all $\ell = 1, \dots, p$.
Since $\xi \in \R^N$ is arbitrary, this implies that 
\begin{equation}\label{SWE_sum}
 \sum_{j: t^j = \hat{t}^\ell} a_j e^{-ix^j\cdot \xi} = 0
\end{equation}
for all $\xi \in \mathbb{R}^N$ and all $\ell = 1, \dots, p$.
Since the points $(t^1, x^1), \dots, (t^n, x^n)$ are distinct, for any fixed $\ell$, 
the $x^j$'s that appear in the sum in \eqref{SWE_sum} are distinct from each other. 
By linear independence
of the functions $e^{-i x^j \cdot \xi}$, we conclude that $a_j = 0$ for all $j$.
\end{proof}

In fact, using the LND property of $u(t, x)$, we can obtain a stronger result 
which says more than the H\"older regularity of the sample functions, 
namely the exact uniform modulus of continuity:

\begin{theorem}\label{thm:MC}
Assume \eqref{beta}.
For any compact interval $I$ in $(0, \infty) \times \R^N$, there exists a constant
$0 < C < \infty$ such that
\begin{equation}\label{eq:MC}
\lim_{\eps \to 0} \sup_{\substack{(t, x), (s, y) \in I,\\0 < |t-s|+|x-y| \le \eps}} 
\frac{|u(t, x) - u(s, y)|}
{(|t-s|+|x-y|)^\alpha\sqrt{\log\big[1+(|t-s|+|x-y|)^{-1}\big]}} 
= C \quad \text{a.s.}
\end{equation}
\end{theorem}

\begin{proof}
Using the Karhunen--Lo\`eve expansion of $u(t, x)$ and Kolmogorov's zero--one law, 
we can show that the limit \eqref{eq:MC} holds for some constant $0 \le C \le \infty$ 
(cf.\ Lemma 7.1.1 of \cite{MR06}).
Then, from \eqref{SWEvar}, we can use the standard metric entropy result 
for Gaussian modulus of continuity \cite{D73} to prove that this limit is finite,
and use Proposition \ref{prop:LND} to prove that it is also strictly positive.
The proof is similar to that of Theorem 3.1 of \cite{LX19} so we omit the details.
\end{proof}

\medskip
\section{Existence of jointly continuous local times}

The objective of this section is to establish the 
existence of jointly continuous local times for the solution of \eqref{SWE}.
Let us first recall a necessary and sufficient condition for the existence of 
square-integrable local times for general Gaussian random fields 
based on the Fourier analytic approach of Berman; see \cite{B69}, \cite[p.36]{GH80}.

Let $X = \{X(z) : z \in T\}$ be an $\R^d$-valued Gaussian random field on a compact 
interval $T \subset \R^k$.
The Fourier transform (or characteristic function) of the occupation measure $\nu_T$ 
of $X$ is
\[ \hat \nu_T(\xi) = \int_{\R^d} e^{i\xi \cdot v} \nu_T(dv) = \int_T e^{i\xi\cdot X(z)} \,dz. \]
By the Plancherel theorem, a necessary and sufficient condition for $X$ to have
a square-integrable local time on $T$, 
namely, $L(\cdot, T) \in L^2(\lambda_d \times \P)$, is
\begin{equation}\label{LT:L2cond}
\int_{\R^d} \int_T \int_T \E[e^{i\xi \cdot (X(z) - X(z'))}] \,dz\,dz'\,d\xi < \infty.
\end{equation}
The integral in \eqref{LT:L2cond} above is equal to 
$\E\int_{\R^d} |\hat \nu_T(\xi)|^2 d\xi$.
In particular, 
when \eqref{LT:L2cond} holds, a version of the local time can be obtained by the 
inverse $L^2$-Fourier transform of $\hat\nu_T$:
\begin{equation}\label{LT:L2}
L(v, T) \overset{L^2}{=} \lim_{M \to \infty} (2\pi)^{-d}
\int_{[-M, M]^d} e^{-i\xi\cdot v} \int_T e^{i\xi \cdot X(z)} dz\,d\xi.
\end{equation}
There are several ways to consider the local times of the stochastic wave equation \eqref{SWE}.
The solution $u(t, x)$ can be regarded as a process in $t$, in $x$, or in $(t, x)$.
Using \eqref{LT:L2cond} and \eqref{SWEvar}, 
we can easily derive the following necessary and sufficient 
conditions for the existence of square-integrable local times for $u$, 
for each of the three cases.

\begin{theorem}\label{thm:LT_E}
Assume \eqref{beta}.
Let $T_1 \subset (0, \infty)$ and $T_2 \subset \R^N$ be compact intervals and
$T = T_1 \times T_2$. 
\begin{enumerate}
\item[(i)] For any fixed $x_0 \in \R^N$,
$\{ u(t, x_0) : t \in T_1\}$ has a square-integrable local time $L^{x_0}(v, T_1)$
on $T_1$ if and only if $\alpha d < 1$.
\item[(ii)] For any fixed $t_0 > 0$,
$\{ u(t_0, x) : x \in T_2\}$ has a square-integrable local time $L_{t_0}(v, T_2)$ on $T_2$
if and only if $\alpha d < N$.
\item[(iii)] $\{ u(t, x) : (t, x) \in T\}$ has a square-integrable local time $L(v, T)$ on $T$
if and only if $\alpha d < 1+N$.
\end{enumerate}
\end{theorem}

By Corollary \ref{cor:LND} and Proposition \ref{prop:LNDx}, 
$u$ satisfies the strong LND property (in the sense of Berman or Pitt) 
in one variable $t$ or $x$ when the other variable is held fixed.
Therefore, if the conditions in (i) and (ii) above hold, 
then the joint continuity and H\"older conditions of the local times follow from
the standard results of \cite{B73, P78, GH80}.
For case (iii), when $N = 1$, $u(t, x)$ satisfies sectorial LND by Corollary 
\ref{cor:sectLND}, so the results of \cite{WX11} can be applied;
otherwise, $u$ satisfies a different type of strong LND which takes an integral form 
by Proposition \ref{prop:LND},
so the standard results of \cite{B73, P78, GH80, X97, WX11} cannot be directly applied.
It can be seen from \eqref{Def:LT} that if $\alpha d < 1$, then
\begin{equation*}
L(v, T) = \int_{T_2} L^x(v, T_1)\, dx \quad \text{a.e. } v,
\end{equation*}
and if $\alpha d < N$, then
\begin{equation*}
L(v, T) = \int_{T_1} L_t(v, T_2)\, dt \quad \text{a.e. } v.
\end{equation*}
While these relations may allow us to deduce regularity of $L(v, T)$ from 
that of $L_t(v, T_2)$ or $L^x(v, T_1)$, 
they are not accessible when $N \le \alpha d < 1+N$.

The main result of this section is Theorem \ref{thm:JC}, which establishes 
joint continuity of the local times of $u$, particularly for case (iii) above.
Our approach is to directly exploit the spherical LND property 
in Proposition \ref{prop:LND} to obtain moment estimates for the local times.

\begin{lemma}\label{lem1}
Let $p > 0$ and $T$ be a compact interval in $(0, \infty) \times \R^N$.
If $\alpha p < 1+N$, then there exists a constant $C < \infty$ such that 
for all intervals $I$ in $T$,
$n \ge 1$ and $(t^1, x^1), \dots, (t^n, x^n) \in I$, 
\begin{equation*}\label{lem1eq}
\qquad \int_{I} dt \,dx \left[ \int_{\mathbb{S}^{N-1}} 
\min_{1\le i \le n} |(t + x\cdot w) - (t^i + x^i \cdot w)|^{2\alpha} \sigma(dw) \right]^{-\frac p 2}
\le C n^{\alpha p} [\lambda_{1+N}(I)]^{1-\frac{\alpha p}{1+N}}.
\end{equation*}
\end{lemma}

\begin{proof}
Fix $(t^1, x^1), \dots, (t^n, x^n) \in I$.
Let $\delta > 0$ be a small constant to be determined.
For $\ell = 1, \dots, N$, let $e_\ell$ denote the unit vector in $\R^N$ 
whose $\ell$-th entry is 1 and all other entries are 0. Let $e_0 = -e_1$.
Also, let $S(e_\ell, \delta) = \{ w \in \mathbb{S}^{N-1} : |w - e_\ell| \le \delta\}$.
Suppose $\delta$ is small enough so that
$S(e_0, \delta), \dots, S(e_N, \delta)$ are disjoint.
For each $0 \le \ell \le N$, 
fix a rotation matrix $R_\ell$ such that $R_\ell e_1 = e_\ell$ and
let $w_\ell = R_\ell w$.
Then
\begin{equation*}
\begin{split}
&\int_{\mathbb{S}^{N-1}} 
\min_{1\le i \le n} |(t + x\cdot w) - (t^i + x^i \cdot w)|^{2\alpha} \sigma(dw)\\
& \ge \sum_{\ell=0}^{N} \int_{S(e_\ell, \delta)}
\min_{1\le i \le n} |(t + x\cdot w) - (t^i + x^i \cdot w)|^{2\alpha} \sigma(dw)\\
& = \int_{S(e_1, \delta)} \sum_{\ell=0}^{N}
\min_{1\le i \le n} |(t + x\cdot w_\ell) - (t^i + x^i \cdot w_\ell)|^{2\alpha} \sigma(dw).
\end{split}
\end{equation*}
Let $M = \sigma(S(e_1, \delta))$.
Since $s \mapsto s^{-p/2}$ is a convex function on $\R_+$, 
by Jensen's inequality,
\begin{equation*}
\begin{split}
&\int_I dt\,dx \left[ \int_{\mathbb{S}^{N-1}}
\min_{1\le i \le n} |(t + x\cdot w) - (t^i + x^i \cdot w)|^{2\alpha} 
\sigma(dw) \right]^{-\frac p 2}\\
&\le M^{-p/2} \int_I dt\, dx \Bigg[ \int_{S(e_1, \delta)}  \sum_{\ell=0}^N
\min_{1\le i \le n} |(t + x\cdot w_\ell) - (t^i + x^i \cdot w_\ell)|^{2\alpha}
\frac{\sigma(dw)}{M}\Bigg]^{-\frac p 2}\\
& \le M^{-p/2} \int_I dt\,dx \int_{S(e_1, \delta)} \Bigg[ \sum_{\ell=0}^N
\min_{1\le i \le n} |(t + x\cdot w_\ell) - (t^i + x^i \cdot w_\ell)|^{2\alpha}\Bigg]^{-\frac p 2} 
\frac{\sigma(dw)}{M}.
\end{split}
\end{equation*}
By using the inequality 
$(\sum_{\ell=0}^N |z_\ell|)^\alpha \le \sum_{\ell=0}^N |z_\ell|^\alpha$ 
for $0 < \alpha < 1$, and Fubini's theorem, this is
\begin{equation}\label{lem1eq2}
\le C \int_{S(e_1, \delta)} \sigma(dw) \int_I dt\,dx  \Bigg[ \sum_{\ell=0}^N
\min_{1\le i \le n} |(t + x\cdot w_\ell) - (t^i + x^i \cdot w_\ell)|^{2}\Bigg]^{-\frac{\alpha p}{2}}.
\end{equation}
For each $w \in S(e_1, \delta)$, we estimate the integral over $I$
using the linear transformation from $\R^{1+N}$ to itself
\[f_w : (t, x) \mapsto y = (y_0, \dots, y_N)\]
defined by $y_\ell = t + x \cdot w_\ell$ for $\ell = 0, \dots, N$.
Write $w_\ell = (w_{\ell,1}, \dots, w_{\ell,N})$
and denote the Jacobian by 
\begin{equation*}
J_w = \det Df_w = \det \begin{pmatrix}
1 & w_{0,1} & \cdots & w_{0, N}\\
1 & w_{1,1} & \cdots & w_{1, N}\\
\vdots & & \ddots & \\
1 & w_{N, 1} & \cdots & w_{N, N}
\end{pmatrix}.
\end{equation*}
Since $w \mapsto J_w$ is continuous and $J_{e_1} = 2$,
we can choose and fix a small enough constant $0 < \delta < 1$ such that 
\begin{equation}\label{Jacobian}
1 \le J_w \le 3 \quad \text{for all } w \in S(e_1, \delta).
\end{equation}
Fix $w \in S(e_1, \delta)$ and let $y^i_\ell = y^i_\ell(w) = t^i + x^i \cdot w_\ell$. 
Then, under the transformation, 
\begin{equation*}
\begin{split}
&\int_I dt\,dx \Bigg[ \sum_{\ell=0}^N
\min_{1\le i \le n} |(t + x\cdot w_\ell) - (t^i + x^i \cdot w_\ell)|^2\Bigg]^{-\frac{\alpha p}{2}}\\
& \le C \int_{f_w(I)} \frac{dy}{{\big(\sum_{\ell=0}^N 
\min\limits_{1 \le i\le n} |y_\ell - y^i_\ell|^2\big)}^{\frac{\alpha p}{2}}}.
\end{split}
\end{equation*}
Consider the Cartesian product $Z = \prod_{\ell=0}^N \{ y^1_\ell, \dots, y^{n}_\ell \}$. 
This set consists of at most $n^{1+N}$ different points in $\R^{1+N}$. 
For each $z = (z_0, \dots, z_N) \in Z$, define 
\[ \Gamma_z = 
\Big\{ y \in f_w(I) : |y_\ell - z_\ell| = \min_{1 \le i \le n} |y_\ell - y^i_\ell| 
\text{ for all } \ell = 0, \dots, N \Big\}. 
\]
Then $\bigcup_{z \in Z} \Gamma_z = f_w(I)$ and the interiors of $\Gamma_z$ are 
non-overlapping, so that
\begin{equation*}
\begin{split}
\int_{f_w(I)} \frac{dy}{{\big(\sum_{\ell=0}^N 
\min\limits_{1 \le i\le n} |y_\ell - y^i_\ell|^2\big)}^{\frac{\alpha p}{2}}}
& = \sum_{z \in Z} \int_{\Gamma_z} \frac{dy}{|y-z|^{\alpha p}}.
\end{split}
\end{equation*}
For each $z \in Z$, we compute the integral over $\Gamma_z$ 
using polar coordinates $y = z + \rho \theta$.
Note that $f_w(I)$ is a convex set in $\R^{1+N}$, and so is $\Gamma_z$. 
Thus, for each $\theta \in \mathbb{S}^N$
the variable $\rho$ takes values between two nonnegative numbers
$\rho_z(\theta) \le \tilde \rho_z(\theta)$.
Let $\sigma(d\theta)$ be the surface measure on $\mathbb{S}^N$.
Then
\begin{equation*}
\begin{split}
\sum_{z \in Z} \int_{\Gamma_z} \frac{dy}{|y-z|^{\alpha p}}
&= \sum_{z \in Z} \int_{\mathbb{S}^{N}} \sigma(d\theta) 
\int_{\rho_z(\theta)}^{\tilde\rho_z(\theta)} \rho^{N-\alpha p} d\rho\\
& = \frac{1}{1+N-\alpha p} \sum_{z \in Z} \int_{\mathbb{S}^{N}} 
[\tilde\rho_z(\theta)^{1+N-\alpha p} - \rho_z(\theta)^{1+N-\alpha p}] \sigma(d\theta)\\
& \le \frac{1}{1+N-\alpha p} \sum_{z \in Z} \int_{\mathbb{S}^{N}} 
{[\tilde\rho_z(\theta)^{1+N} - \rho_z(\theta)^{1+N}]}^{1-\frac{\alpha p}{1+N}}
\sigma(d\theta).
\end{split}
\end{equation*}
The last inequality follows from $b^q - a^q \le (b-a)^q$
for $0 \le a \le b$ and $0 < q < 1$, which can be verified easily.
Since the Lebesgue measure of $\Gamma_z$ is
\begin{equation*}
\lambda_{1+N}(\Gamma_z) 
= \frac{1}{1+N} \int_{\mathbb{S}^{N}}[\tilde \rho_z(\theta)^{1+N} - \rho_z(\theta)^{1+N}] 
\sigma(d\theta)
\end{equation*}
and the function $s \mapsto s^{1-\frac{\alpha p}{1+N}}$ is concave on $\R_+$,
we can use Jensen's inequality to get that
\begin{equation*}
\begin{split}
&\sum_{z \in Z} \int_{\mathbb{S}^{N}} 
[\tilde \rho_z(\theta)^{1+N} - \rho_z(\theta)^{1+N}]^{1- \frac{\alpha p}{1+N}} 
\sigma(d\theta)\\
& \le \sum_{z \in Z} \bigg(\int_{\mathbb{S}^{N}} 
[\tilde \rho_z(\theta)^{1+N} - \rho_z(\theta)^{1+N}] \sigma(d\theta) 
\bigg)^{1-\frac{\alpha p}{1+N}}\\
& = {(1+N)}^{1-\frac{\alpha p}{1+N}} 
\sum_{z \in Z} \big(\lambda_{1+N}(\Gamma_z)\big)^{1-\frac{\alpha p}{1+N}}.
\end{split}
\end{equation*}
Let $|Z|$ denote the cardinality of $Z$. Then by Jensen's inequality again, this is
\begin{equation*}
\begin{split}
&= {(1+N)}^{1-\frac{\alpha p}{1+N}} |Z| \cdot \frac{1}{|Z|}\sum_{z \in Z} 
\big(\lambda_{1+N}(\Gamma_z)\big)^{1-\frac{\alpha p}{1+N}}\\
& \le C|Z| \bigg( \frac{1}{|Z|}\sum_{z \in Z} 
\lambda_{1+N}(\Gamma_z) \bigg)^{1-\frac{\alpha p}{1+N}}\\
& = C{|Z|}^{\frac{\alpha p}{1+N}} {[\lambda_{1+N}(f_w(I))]}^{1-\frac{\alpha p}{1+N}}.
\end{split}
\end{equation*}
Since $|Z| \le n^{1+N}$ and $\lambda_{1+N}(f_w(I)) \le C \lambda_{1+N}(I)$ by \eqref{Jacobian},
we deduce that
\begin{equation*}
\int_I dt\,dx \Bigg[ \sum_{\ell=0}^N
\min_{1\le i \le n} |(t+ x\cdot w_\ell) - (t^i + x^i \cdot w_\ell)|^{2\alpha}\Bigg]^{-p/2}
\le C n^{\alpha p} {[\lambda_{1+N}(I)]}^{1-\frac{\alpha p}{1+N}},
\end{equation*}
where $C$ is a constant independent of $w \in S(e_1, \delta)$.
Then put this back into \eqref{lem1eq2} to complete the proof.
\end{proof}

The proof of Lemma \ref{lem1} also yields the following result.

\begin{lemma}\label{lem1'}
Let $T$ be a compact interval in $\R^k$.
Let $p > 0$ be such that $\alpha p < k$.
Then there exists a finite constant $C$ such that for all convex subsets $F$ of $T$,
for all $n \ge 1$ and $y, y^1, \dots, y^n \in F$,
\[ \int_F \frac{dy}{\big( \sum_{\ell=1}^k \min\limits_{1\le i \le n} |y_\ell - y^i_\ell|^2 
\big)^{\frac{\alpha p}{2}}} \le C n^{\alpha p} [\lambda_k(F)]^{1-\frac{\alpha p}{k}}. \]
\end{lemma}

For any mean-zero Gaussian vector $(X_1, \dots, X_n)$,
the following formula can be easily verified:
\begin{equation}\label{detcov} 
\det \mathrm{Cov}(X_1, \dots, X_n) = \mathrm{Var}(X_1) 
\prod_{j=2}^n \mathrm{Var}(X_j| X_1, \dots, X_{j-1}),
\end{equation}
where $\det \mathrm{Cov}(X_1, \dots, X_n)$ denotes the determinant of the covariance 
matrix of $(X_1, \dots, X_n)$.

\begin{lemma}\label{lem:fBm}
Let $I$ be a compact interval in $\R$ and $t^0 = 0$.
\begin{enumerate}
\item[(i)] If $I \subset (0, \infty)$ and $B(t)$ is a fractional Brownian motion 
with Hurst index $0 < \alpha < 1$, then
there exist constants $0 < C_1 \le C_2 < \infty$ such that for any $n \ge 1$,
for any $t, t^1, \dots, t^n \in I$,
\[ C_2 \min_{0\le i \le n} |t - t^i|^{2\alpha} \le \mathrm{Var}(B(t)|B(t^1), \dots, B(t^n))
\le C_2 \min_{0\le i \le n} |t - t^i|^{2\alpha}. \]
\item[(ii)] There exists a constant $C > 0$ such that 
for any $n \ge 1$, for any $t^1, \dots, t^n \in I$, 
for any permutation $\pi$ on $\{1, \dots, n\}$, 
\[ \prod_{j=2}^n \min_{1 \le i \le j-1}|t^{\pi(j)} - t^{\pi(i)}| 
\ge C^n \prod_{j=2}^n \min_{1 \le i \le j-1}|t^j - t^i|. \]
\end{enumerate}
\end{lemma}

\begin{proof}
(i). The first inequality is due to the strong LND property 
of the fractional Brownian motion \cite[Lemma 7.1]{P78};
the second inequality holds because the conditional variance is 
$\le \mathrm{Var}(B(t))$ and 
is $\le \mathrm{Var}(B(t) - B(t^i))$ for every $1 \le i \le n$.

(ii). Clearly, both sides of the inequality are translation invariant, so by shifting
we may assume that $I \subset (0, \infty)$ and $t > \mathrm{diam}(I)$ for all $t \in I$. 
Take $\alpha = 1/2$.
Since $\det\mathrm{Cov}(B(t^{\pi(1)}), \dots, B(t^{\pi(n)}))
= \det\mathrm{Cov}(B(t^1), \dots, B(t^n))$,
the result follows from part (i) of this lemma and the formula \eqref{detcov}.
\end{proof}

\begin{lemma}\label{lem:CD82}\cite[Lemma 2]{CD82}
Let $Z_1, \dots, Z_n$ be mean-zero Gaussian random variables that are 
linearly independent. 
Let $g: \R \to \R$ be a measurable function such that 
$\int_\R g(x) e^{-\eps x^2} dx < \infty$ for every $\eps > 0$. Then
\begin{equation*}
\begin{split}
&\int_{\R^n} g(\xi_1) \exp\bigg[ 
{-\frac 1 2 \mathrm{Var}\Big( \sum_{j=1}^n \xi_j Z_j\Big)} \bigg] d\xi_1 \cdots d\xi_n\\
& = \frac{(2\pi)^{(n-1)/2}}{[\det \mathrm{Cov}(Z_1, \dots, Z_n)]^{1/2}}
\int_\R g\Big(\frac{x}{V_1}\Big) \,e^{-x^2/2} \, dx,
\end{split}
\end{equation*}
where $V_1^2 = \mathrm{Var}(Z_1|Z_2, \dots, Z_n)$.
\end{lemma}

\begin{lemma}\label{lem:J}
Let $T$ be a compact interval in $(0, \infty) \times \R^N$.
Let $q_{j,k} \ge 0$ and $q > 0$ be such that 
$\alpha (d+2q) < 1+N$ and $\sum_{k=1}^d q_{j,k} = q$ for each $j$.
For $\bar z = (z^1, \dots, z^n) \in T^n$, let
\begin{equation*}
J(\bar z) = \int_{\R^{nd}} \Big(\prod_{j=1}^n \prod_{k=1}^d {|\xi_k^j|}^{q_{j,k}}\Big)
\E\big( e^{i\sum_{j=1}^n \sum_{k=1}^d \xi_k^j u_k(z^j)} \big) \, d\bar \xi.
\end{equation*}
where $\bar \xi = (\xi^1_1, \dots, \xi^n_d)$.
Then there exist constants $C< \infty$ and $r_0 > 0$ such that the following hold
for all $n \ge 2$:
\begin{enumerate}
\item[(i)] For all compact intervals $I \subset T$ with side lengths $\le r_0$, 
\begin{equation}\label{lemJ1}
\int_{I^n} J(\bar z) \, d\bar z \le C^n {(n!)}^{\alpha d + (\frac 1 2 +2\alpha)q} 
{[\lambda_{1+N}(I)]}^{n(1 - \frac{\alpha (d+2q)}{1+N})\frac{d}{d+2q}}.
\end{equation}
\item[(ii)] If, in addition, $I$ has side lengths $\le r$ with $0 < r \le r_0$, then
\begin{equation}\label{lemJ2}
\int_{I^n} J(\bar z) \, d\bar z \le C^n {(n!)}^{\alpha d + (\frac 1 2+\alpha)q} 
{r}^{n(1+N - \alpha (d+q))}.
\end{equation}
\end{enumerate}
\end{lemma}

\begin{proof}
(i). Let $I \subset T$ be a compact interval with side lengths $\le r_0$.
By the fact that $u_1, \dots, u_d$ are i.i.d.\ and Gaussian,
and by the generalized H\"older inequality,
\begin{equation*}
\begin{split}
J(\bar z) &= \prod_{k=1}^d \int_{\R^n} \prod_{j=1}^n {|\xi_k^j|}^{q_{j,k}}
\exp\bigg[ -\frac 1 2 \mathrm{Var}\Big(\sum_{j=1}^n \xi_k^j u_1(z^j)\Big) \bigg] d\bar \xi_k\\
&\le \prod_{k=1}^d \prod_{j=1}^n \bigg\{ \int_{\R^n} {|\xi_k^j|}^{nq_{j,k}}
\exp\bigg[ -\frac 1 2 \mathrm{Var}\Big(\sum_{j=1}^n \xi_k^j u_1(z^j)\Big) \bigg] d\bar \xi_k\bigg\}^{\frac1 n},
\end{split}
\end{equation*}
where $\bar \xi_k = (\xi^1_k, \dots, \xi^n_k) \in \R^n$.
It is enough to consider points $z^1, \dots, z^n \in I$ that are distinct from each other
since the set of such points has full Lebesgue measure in $I^n$.
Then $u_1(z^1), \dots, u_1(z^n)$ are linearly independent by Proposition \ref{lem:LI}.
By Lemma \ref{lem:CD82} and Stirling's formula, $J(\bar z)$ is bounded by
\begin{align}\label{J(x)}
\notag
&\quad C^n\prod_{k=1}^d \prod_{j=1}^n
\Big\{{[\det \mathrm{Cov}(u_1(z^1), \dots, u_1(z^n))]}^{-\frac{1}{2}}
{[\mathrm{Var}(u_1(z^j)|u_1(z^i) : i\ne j )]}^{-\frac{nq_{j,k}}{2}}
\Gamma\Big(\frac{nq_{j,k}+1}{2}\Big)\Big\}^{\frac 1 n}\\
&\le C^n {(n!)}^{\frac q 2} {[\det \mathrm{Cov}(u_1(z^1), \dots, u_1(z^n))]}^{-\frac d 2} \,
\prod_{j=1}^n{[\mathrm{Var}(u_1(z^j)|u_1(z^i) : i \ne j)]}^{-\frac q 2}.
\end{align}
Define $e_0, \dots, e_N, w_0, \dots, w_N$ and $\delta$ as in the proof of Lemma \ref{lem1}.
By \eqref{detcov} and Proposition \ref{prop:LND}, for $r_0$ small enough,
\begin{align*}
&{[\det \mathrm{Cov}(u_1(z^1), \dots, u_1(z^n))]}^{\frac d 2}
\prod_{j=1}^n {[\mathrm{Var}(u_1(z^j)|u_1(z^i) : i \ne j)]}^{\frac q 2}\\
& \ge C^n
\prod_{j=1}^n \bigg[ \int_{\mathbb{S}^{N-1}}r_j(w)^{2\alpha} 
\sigma(dw) \bigg]^{\frac d 2} \prod_{j=1}^n 
\bigg[\int_{\mathbb{S}^{N-1}} \tilde r_j(w)^{2\alpha} \sigma(dw) \bigg]^{\frac q 2}\\
& \ge C^n \prod_{j=1}^n \bigg[ \int_{S(e_1, \delta)} 
\sum_{\ell=0}^N r_j(w_\ell)^{2\alpha} 
\sigma(dw) \bigg]^{\frac d 2} \prod_{j=1}^n 
\bigg[\int_{S(e_1, \delta)} \sum_{\ell=0}^N\tilde r_j(w_\ell)^{2\alpha} \sigma(dw) 
\bigg]^{\frac q 2},
\end{align*}
where $r_1(w) \equiv 1$,
\begin{align*}
r_j(w) &= \min_{1 \le i \le j-1} |(t^j + x^j \cdot w) - (t^i + x^i \cdot w)|, \quad 2 \le j \le n,\\
\tilde{r}_j(w) &= \min_{1 \le i \le n, \,i \ne j} |(t^j + x^j \cdot w) - (t^i + x^i \cdot w)|, 
\quad 1\le j \le n.
\end{align*}
Then, by the generalized H\"older inequality,
\begin{align*}
&{[\det \mathrm{Cov}(u_1(z^1), \dots, u_1(z^n))]}^{\frac d 2}
\prod_{j=1}^n {[\mathrm{Var}(u_1(z^j)|u_1(z^i) : i \ne j)]}^{\frac q 2}\\
& \ge C^n \bigg[\int_{S(e_1, \delta)} \prod_{j=1}^n 
\Big(\sum_{\ell=0}^N r_j(w_\ell)^{2\alpha}\Big)^{\frac{d}{2m}} \Big(
\sum_{\ell=0}^N \tilde r_j(w_\ell)^{2\alpha}\Big)^{\frac{q}{2m}} \sigma(dw) \bigg]^m,
\end{align*}
where $m = \frac{n(d+q)}{2}$.
Recall that $\delta$ is a constant and $M = \sigma(S(e_1, \delta))$. 
Then, by Jensen's inequality for the convex function $x \mapsto x^{-m}$ on $\R_+$, 
we have
\begin{align}\label{J:eq2}
\begin{split}
&{[\det \mathrm{Cov}(u_1(z^1), \dots, u_1(z^n))]}^{-\frac d 2}
\prod_{j=1}^n {[\mathrm{Var}(u_1(z^j)|u_1(z^i) : i \ne j)]}^{-\frac q 2}\\
& \le C^n M^{-m-1} \int_{S(e_1, \delta)}
\prod_{j=1}^n\bigg[ \Big( \sum_{\ell=0}^N r_j(w_\ell)^{2\alpha} \Big)^{-\frac{d}{2}}
\Big( \sum_{\ell=0}^N \tilde r_j(w_\ell)^{2\alpha} \Big)^{-\frac{q}{2}}\bigg]
\sigma(dw)\\
& \le C^n \int_{S(e_1, \delta)} 
\prod_{j=1}^n\bigg[ \Big( \sum_{\ell=0}^N r_j(w_\ell)^2\Big)^{-\frac{\alpha d}{2}}
\Big(\sum_{\ell=0}^N \tilde r_j(w_\ell)^2\Big)^{-\frac{\alpha q}{2}} \bigg] \sigma(dw).
\end{split}
\end{align}
Recall the transformation $f_w: z = (t, x) \mapsto y = (y_0, \dots, y_N)$ defined by
$y_\ell = t + x \cdot w_\ell$ and that it satisfies \eqref{Jacobian}.
To estimate the integral of $J(\bar z)$ over $I^n$, first use \eqref{J(x)} 
and \eqref{J:eq2}. Then, by interchanging the order of integration and 
using the transformation, 
followed by H\"older's inequality with exponents 
$\frac{d+2q}{d}$ and $\frac{d+2q}{2q}$, we get that
\begin{align*}
&\int_{I^n} J(\bar z) \, d\bar z\\
& \le C^n {(n!)}^{\frac q 2} \int_{S(e_1, \delta)} \sigma(dw) 
\int_{{[f_w(I)]}^n} \frac{d y^1 \cdots d y^n}{\prod_{j=1}^n\big[ \big( \sum_{\ell=0}^N 
\min\limits_{1 \le i \le j-1} |y^j_\ell - y^i_\ell|^2\big)^{\frac{\alpha d}{2}}\,
\big(\sum_{\ell=0}^N\min\limits_{i:\,i \ne j}|y^j_\ell - y^i_\ell|^2\big)^{\frac{\alpha q}{2}}
\big]}\\
& \le C^n {(n!)}^{\frac q 2} \int_{S(e_1, \delta)} A_1(w) A_2(w) \sigma(dw),
\end{align*}
where $y^j_\ell = t^j + x^j \cdot w_\ell$,
\begin{align*}
A_1(w) &= \Bigg\{ \int_{{[f_w(I)]}^n} \frac{dy^1\cdots dy^n}
{\prod_{j=1}^n \big( \sum_{\ell=0}^N \min\limits_{1 \le i \le j-1} 
|y^j_\ell - y^i_\ell|^2\big)^{\frac{\alpha (d+2q)}{2}}}\Bigg\}^{\frac{d}{d+2q}},\\ 
A_2(w) &= \Bigg\{ \int_{{[f_w(I)]}^n} \frac{dy^1\cdots dy^n}
{\prod_{j=1}^n \big( \sum_{\ell=0}^N \min\limits_{i:\, i \ne j} |y^j_\ell - y^i_\ell|^2\big)^{\frac{\alpha (d+2q)}{4}}}\Bigg\}^{\frac{2q}{d+2q}}.
\end{align*}
Now, we need the assumption that $\alpha (d+2q) < 1+N$, and recall that \eqref{Jacobian} implies
$\lambda_{1+N}(f_w(I)) \le C\lambda_{1+N}(I)$ for all $w \in S(e_1, \delta)$.
Then by Lemma \ref{lem1'}, we have
\begin{equation*}\label{A1}
A_1(w) \le C^n (n!)^{\alpha d} 
{[\lambda_{1+N}(I)]}^{n(1 - \frac{\alpha(d+2q)}{1+N})\frac{d}{d+2q}}.
\end{equation*}
For $A_2$, we first use the AM--GM inequality to get
\begin{align*}
A_2(w) &\le \Bigg\{ \int_{{[f_w(I)]}^n} \frac{dy^1 \cdots dy^n}{\prod_{j=1}^n
\prod_{\ell=0}^N\min\limits_{i:\, i \ne j}|y_\ell^j - y_\ell^i|^{\frac{\alpha(d+2q)}{2(1+N)}}}
\Bigg\}^{\frac{2q}{d+2q}}.
\end{align*}
Since $I$ has side lengths $\le r_0$, we can see from the definition of $y^j_\ell$ that 
each $y^j = (y^j_0, \dots, y^j_N)$ is contained in 
$\prod_{\ell=0}^N \tilde I_\ell$, where each $\tilde I_\ell$ 
is an interval in $\R$ of length $\le (1+N)r_0$. From this, we get
\begin{align*}
& \le \prod_{\ell=0}^N \Bigg\{ \int_{{(\tilde I_\ell)}^n} \frac{dy^1_\ell \cdots dy^n_\ell}
{\prod_{j=1}^n \min\limits_{i:\, i \ne j}|y_\ell^j - y_\ell^i|^{\frac{\alpha(d+2q)}{2(1+N)}}}
\Bigg\}^{\frac{2q}{d+2q}}.
\end{align*}
Fix $\ell$. For each $(y^1_\ell, \dots, y^n_\ell) \in (\tilde I_\ell)^n$,
let $\pi$ be a permutation such that
$y^{\pi(1)}_\ell \le \cdots \le y^{\pi(n)}_\ell$, and
note that the $y^j_\ell$ are all bounded. 
For convenience, set $y_\ell^{\pi(0)} = y_\ell^{\pi(n+1)} = 0$.
It follows that
\begin{align*}
\prod_{j=1}^n \min\limits_{i:\, i \ne j}|y_\ell^j - y_\ell^i|
& = \prod_{j=1}^n \min\limits_{i:\, i \ne j}|y_\ell^{\pi(j)} - y_\ell^{\pi(i)}|\\
& = \prod_{j=1}^n \min\{ |y_\ell^{\pi(j)} - y_\ell^{\pi(j-1)}|, 
|y_\ell^{\pi(j)} - y_\ell^{\pi(j+1)}|\}\\
& \ge C^n \prod_{j=1}^n \big(|y_\ell^{\pi(j)} - y_\ell^{\pi(j-1)}|\cdot
|y_\ell^{\pi(j)} - y_\ell^{\pi(j+1)}|\big)\\
& \ge C^n \prod_{j=2}^n |y^{\pi(j)}_\ell - y^{\pi(j-1)}_\ell|^2\\
& \ge C^n \prod_{j=2}^n \min_{1 \le i \le j-1}|y^j_\ell - y^i_\ell|^2.
\end{align*}
The last inequality follows from Lemma \ref{lem:fBm}(ii).
Then, by Lemma \ref{lem1'},
\begin{equation*}\label{A2}
\begin{split}
A_2(w) & \le C^n
\prod_{\ell=0}^N \Bigg\{ \int_{{(\tilde I_\ell)}^n} \frac{dy^1_\ell \cdots dy^n_\ell}
{\prod_{j=2}^n \min\limits_{1 \le i \le j-1}|y_\ell^j - y_\ell^i|^{\frac{\alpha(d+2q)}{1+N}}}
\Bigg\}^{\frac{2q}{d+2q}}
\le C^n (n!)^{2\alpha q}.
\end{split}
\end{equation*}
The constant $C$ does not depend on $w \in S(e_1, \delta)$.
This leads to \eqref{lemJ1}.

To prove (ii), suppose that $I$ has side lengths $\le r$, where $r \le r_0$.
Again, by \eqref{detcov}, Proposition \ref{prop:LND} and 
the generalized H\"older inequality, for $r_0$ small enough,
\begin{equation*}
\begin{split}
&{[\det \mathrm{Cov}(u_1(z^1), \dots, u_1(z^n))]}^{\frac d 2} 
\prod_{j=1}^n{[\mathrm{Var}(u_1(z^j)|u_1(z^i) : i \ne j)]}^{\frac q 2}\\
& \ge C^n \prod_{j=1}^n \bigg[\int_{\mathbb{S}^{N-1}} r_j(w)^{2\alpha} 
\sigma(dw) \bigg]^{\frac d 2} \prod_{j=1}^n 
\bigg[\int_{\mathbb{S}^{N-1}} \tilde r_j(w)^{2\alpha} \sigma(dw) \bigg]^{\frac q 2}\\
& \ge C^n \bigg[ \int_{\mathbb{S}^{N-1}} \prod_{j=1}^n 
\Big({r_j(w)}^{\frac{\alpha d}{m}}\, {\tilde r_j(w)}^{\frac{\alpha q}{m}} \Big)\,
\sigma(dw) \bigg]^m\\
& \ge C^n \bigg[  \int_{S(e_1, \delta)}\sum_{\ell=0}^N \prod_{j=1}^n 
\Big({r_j(w_\ell)}^{\frac{\alpha d}{m}}\, {\tilde r_j(w_\ell)}^{\frac{\alpha q}{m}} \Big)\,
\sigma(dw) \bigg]^m,\\
\end{split}
\end{equation*}
where $m = \frac{n(d+q)}{2}$, $r_j$ and $\tilde r_j$ are defined as before.
Define the variables $y^j_\ell = t^j + x^j \cdot w_\ell$ as before.
Then, by the AM--GM inequality and Jensen's inequality,
\begin{align*}
&{[\det \mathrm{Cov}(u_1(z^1), \dots, u_1(z^n))]}^{-\frac d 2} \,
\prod_{j=1}^n{[\mathrm{Var}(u_1(z^j)|u_1(z^i) : i \ne j)]}^{-\frac q 2}\\
& \le C^n \bigg[ \int_{S(e_1, \delta)} \prod_{\ell=0}^N \prod_{j=1}^n 
\Big({r_j(w_\ell)}^{\frac{\alpha d}{m(1+N)}}\, {\tilde r_j(w_\ell)}^{\frac{\alpha q}{m(1+N)}}
\Big) \sigma(dw) \bigg]^{-m}\\
& \le C^n \int_{S(e_1, \delta)} \prod_{\ell=0}^N \prod_{j=1}^n 
\Big({r_j(w_\ell)}^{-\frac{\alpha d}{1+N}}\, {\tilde r_j(w_\ell)}^{-\frac{\alpha q}{1+N}}
\Big) \sigma(dw).
\end{align*}
Since $I$ has side lengths $\le r$, each $y^j_\ell$ is contained in an interval 
$\tilde I_\ell \subset \R$ of length $\le (1+N)r$.
Then, using \eqref{J(x)} and the transformation $f_w: z \mapsto y$, we have
\begin{equation}\label{intJ1}
\begin{split}
\int_{I^n} J(\bar z)\, d\bar z & \le C^n {(n!)}^{\frac q 2} 
\int_{S(e_1, \delta)} \sigma(dw) \times\\
&\qquad \qquad \prod_{\ell=0}^N \int_{{(\tilde I_\ell)}^n} \frac{dy^1_\ell \cdots dy^n_\ell}
{\prod_{j=2}^n \min\limits_{1 \le i \le j-1}|y^j_\ell - y^i_\ell|^{\frac{\alpha d}{1+N}}
\prod_{j=1}^n\min\limits_{i:\, i \ne j}|y^j_\ell - y^i_\ell|^{\frac{\alpha q}{1+N}}}.
\end{split}
\end{equation}
Fix $\ell$ and consider the integral over $(\tilde I_\ell)^n$.
For $(y^1_\ell, \dots, y^n_\ell) \in (\tilde I_\ell)^n$,
let $\pi$ be a permutation such that $y^{\pi(1)}_\ell \le \dots \le y^{\pi(n)}_\ell$.
Then, by Lemma \ref{lem:fBm}(ii), 
\begin{align*}
&\quad \prod_{j=2}^n \min\limits_{1 \le i \le j-1}{|y^j_\ell - y^i_\ell|}^{\frac{\alpha d}{1+N}} 
\prod_{j=1}^n \min\limits_{i:\,i \ne j}{|y^j_\ell - y^i_\ell|}^{\frac{\alpha q}{1+N}}\\
& \ge C^n \prod_{j=2}^n
\min\limits_{1 \le i \le j-1}{|y^{\pi(j)}_\ell - y^{\pi(i)}_\ell|}^{\frac{\alpha d}{1+N}} 
\prod_{j=1}^n \min\limits_{i:\,i \ne j}{|y^{\pi(j)}_\ell - y^{\pi(i)}_\ell|}^{\frac{\alpha q}{1+N}}
\\
& \ge C^n \prod_{j=1}^n \Big( {|y^{\pi(j)}_\ell - y^{\pi(j-1)}_\ell|}^{\frac{\alpha d}{1+N}} 
{|y^{\pi(j)}_\ell - y^{\pi(j-1)}_\ell|}^{\frac{\alpha q}{1+N}\theta_j}
{|y^{\pi(j)}_\ell - y^{\pi(j+1)}_\ell|}^{\frac{\alpha q}{1+N}(1-\theta_j)}\Big)
\end{align*}
for some $\theta = (\theta_1, \dots, \theta_n) \in \{0, 1\}^n$ with 
$\theta_1 = 0$ and $\theta_n = 1$. Denote $\theta'_j = 1-\theta_j$.
By Lemma \ref{lem:fBm}(ii) again, we get
\begin{align*}
& \ge C^n \prod_{j=2}^n
|y^{\pi(j)}_\ell - y^{\pi(j-1)}_\ell|^{\frac{\alpha}{1+N}(d + q(\theta_j+\theta'_{j-1}))}\\
& \ge C^n \prod_{j=2}^n\min_{1 \le i \le j-1}
{|y^j_\ell - y^i_\ell|}^{\frac{\alpha}{1+N}(d + q(\theta_j+\theta'_{j-1}))}.
\end{align*}
Hence, for each $\ell$, the integral over $(\tilde I_\ell)^n$ in \eqref{intJ1} is bounded by
\begin{align*}
& C^n \sum_{\theta} \int_{{(\tilde I_\ell)}^n}\frac{dy^1_\ell \cdots dy^n_\ell}
{\prod_{j=2}^n\min\limits_{1 \le i \le j-1}
{|y^j_\ell - y^i_\ell|}^{\frac{\alpha}{1+N}(d + q(\theta_j +\theta'_{j-1}))}}.
\end{align*}
The sum runs over all $\theta \in \{0, 1\}^n$ 
with $\theta_1 = 0$ and $\theta_n = 1$, containing $< 2^n$ summands.
Note that
$\frac{\alpha}{1+N}(d+q (\theta_j +\theta'_{j-1})) \le \frac{\alpha}{1+N} (d+2q) < 1$.
By Lemma \ref{lem1'} and the relation $\theta_j + \theta'_j = 1$, we get
\begin{align*}
&\le C^n \sum_\theta \prod_{j=1}^n 
\Big(j^{\frac{\alpha}{1+N}(d+q(\theta_j + \theta'_{j-1}))} 
r^{1-\frac{\alpha}{1+N}(d+q(\theta_j+\theta'_{j-1}))}\Big)\\
& \le C^n r^{n(1-\frac{\alpha (d+q)}{1+N})} \sum_\theta
\prod_{j=1}^n j^{\frac{\alpha}{1+N}(d+q\theta_j)}
\prod_{j=1}^n (2j)^{\frac{\alpha}{1+N}(d+q\theta'_j)}\\
&\le C^n (n!)^{\frac{\alpha(d + q)}{1+N}} r^{n(1-\frac{\alpha (d+q)}{1+N})}
\end{align*}
(we have set $\theta'_0 = 0$ in the above).
Finally, put this back into \eqref{intJ1} to conclude \eqref{lemJ2}.
\end{proof}

We can use the above lemmas to get moment estimates for the local times.
Recall the following formulas for the local times $L(v, T)$ of an $\R^d$-valued 
random field $\{ X(z) : z \in T\}$, which can be found in \cite[\S 25]{GH80}:
for any even number $n \ge 2$, for any $v, \tilde v \in \R^d$,
\begin{align}
\label{Eq:M-LT}
\E[L(v, T)^n] &= (2\pi)^{-nd} \int_{T^n} d\bar z \int_{\R^{nd}} d\bar \xi\,
e^{- i \sum_{j=1}^n \xi^j \cdot v} \,\E\big(e^{i\sum_{j=1}^n \xi^j \cdot X(z^j)}\big),\\
\label{Eq:M-LTI}
\E[(L(v, T) - L(\tilde v, T))^n] &= (2\pi)^{-nd} \int_{T^n} d\bar z \int_{\R^{nd}} d\bar \xi\,
\prod_{j=1}^n\big(e^{- i \xi^j \cdot v}  - e^{- i \xi^j \cdot \tilde v}\big) \, 
\E\big(e^{i\sum_{j=1}^n \xi^j \cdot X(z^j)}\big).
\end{align}

\begin{proposition}\label{prop:M-LT}
Assume \eqref{beta} and $\alpha d < 1+N$.
Let $T$ be a compact interval in $(0, \infty) \times \R^N$
and $L(v, T)$ be the local time of $\{u(t, x) : (t, x) \in T\}$.
Then the following statements hold for some constant $r_0 > 0$:
\begin{enumerate}
\item[(i)] There exists a constant $C$ such that for all intervals $I$ in $T$ with side 
lengths $\le r_0$, for all $v \in \R^d$, for all even numbers $n \ge 2$,
\begin{equation}\label{M-LT:eq1}
\E[L(v, I)^n] \le C^n {(n!)}^{\alpha d} {[\lambda_{1+N}(I)]}^{n(1 - \frac{\alpha d}{1+N})}.
\end{equation}
\item[(ii)]
For any $0 < \gamma <\min\{\frac 1 2 (\frac{1+N}{\alpha} - d), 1\}$,
there exists a constant $C$ such that
for all intervals $I$ in $T$ with side lengths $\le r$, where $0 < r \le r_0$, 
for all $v, \tilde v \in \R^d$, for all even numbers $n \ge 2$,
\begin{equation}\label{M-LT:eq2}
\E[(L(v, I) - L(\tilde v, I))^n] \le 
C^n |v-\tilde v|^{n\gamma} {(n!)}^{\alpha d + (\frac 1 2 +\alpha)\gamma} 
{r}^{n(1+N-\alpha (d + \gamma))}.
\end{equation}
\end{enumerate}
\end{proposition}

\begin{proof}
(i). Write $z = (t, x)$. By \eqref{Eq:M-LT},
\begin{equation*}
\begin{split}
\E[L(v, I)^n] &\le (2\pi)^{-nd} 
\int_{I^n}  \int_{\R^{nd}} \E[e^{i\sum_{j=1}^n \xi^j \cdot u(z^j)}] \,d\bar \xi\,d\bar z\\
&= (2\pi)^{-nd/2} 
\int_{I^n} [\det \mathrm{Cov}(u_1(z^1), \dots, u_1(z^n))]^{-d/2} \, 
dz^1 \cdots dz^n.
\end{split}
\end{equation*}
By \eqref{detcov} and Proposition \ref{prop:LND}, for $r_0$ sufficiently small, this is
\begin{equation*}
\le C^n \int_{I^n} \prod_{j=2}^n \bigg[\int_{\mathbb{S}^{N-1}} \min_{1 \le i \le j-1} 
|(t^j+ x^j \cdot w) - (t^i + x^i\cdot w)|^{2\alpha}\sigma(dw)\bigg]^{-d/2} \, 
dz^1 \cdots dz^n.
\end{equation*}
Then, integrate in the order $dz^n, dz^{n-1}, \dots, dz^1$
and apply Lemma \ref{lem1}(i) repeatedly to get \eqref{M-LT:eq1}.

(ii). By \eqref{Eq:M-LTI},
\begin{equation*}
\begin{split}
\E[(L(v, I) - L(\tilde v, I))^n]
& \le (2\pi)^{-nd} \int_{I^n} d\bar z \int_{\R^{nd}} d\bar\xi \,
\prod_{j=1}^n \big|e^{-i\xi^j \cdot v} - e^{-i\xi^j \cdot \tilde v}\big| \,
\E\big(e^{i\sum_{j=1}^n \xi^j \cdot u(z^j)} \big).
\end{split}
\end{equation*}
For any $0 < \gamma < 1$, we have 
the inequality $|e^{-ix} - e^{-iy}| \le 2 |x-y|^\gamma$,
which implies that
\[ \prod_{j=1}^n \big|e^{-i\xi^j \cdot v} - e^{-i\xi^j \cdot \tilde v}\big|
\le 2^n |v - \tilde v|^{n\gamma} \sum_{(k_1, \dots, k_n)} 
\prod_{j=1}^n {\big|\xi^j_{k_j}\big|}^\gamma, \]
where the sum is taken over all $(k_1, \dots, k_n) \in \{1, \dots, d\}^n$.
Thus,
\begin{equation*}
\begin{split}
\E[(L(v, I) - L(\tilde v, I))^n]
& \le C^n |v - \tilde v|^{n\gamma} \sum_{(k_1, \dots, k_n)}
 \int_{I^n} d\bar z \int_{\R^{nd}} d\bar\xi \,
\Big(\prod_{j=1}^n {|\xi^j_{k_j}|}^\gamma \Big) \,
\E\big(e^{i\sum_{j=1}^n \xi^j \cdot u(z^j)} \big).
\end{split}
\end{equation*}
Let $\gamma$ satisfy $\alpha (d+ 2\gamma) < 1+N$.
Then we can derive \eqref{M-LT:eq2} using Lemma \ref{lem:J}(ii)
with $q_{j, k} = \gamma$ if $k = k_j$, and $q_{j, k} = 0$ otherwise.
\end{proof}

We now conclude the main result of this section.

\begin{theorem}\label{thm:JC}
Assume \eqref{beta}.
\begin{enumerate}
\item[(i)] If $\alpha d < 1$, then for any fixed $x_0 \in \R^N$,
$\{ u(t, x_0) : t \in T_1 \}$ has a jointly continuous local time on 
any compact interval $T_1$ in $(0, \infty)$.
\smallskip
\item[(ii)] If $\alpha d < N$, then for any fixed $t_0 > 0$, 
$\{ u(t_0, x) : x \in T_2 \}$ has a jointly continuous local time on any compact 
interval $T_2$ in $\R^N$.
\smallskip
\item[(iii)] If $\alpha d < 1+N$, then $\{ u(t, x) : (t, x) \in T\}$ has a jointly continuous 
local time on any compact interval $T$ in $(0, \infty) \times \R^N$.
\end{enumerate}
\end{theorem}

\begin{proof}
(i) and (ii). By Corollary \ref{cor:LND} and Proposition \ref{prop:LNDx}, 
the processes $t \mapsto u(t, x_0)$ and $x \mapsto u(t_0, x)$ satisfy the LND property
in the sense of Berman or Pitt, so the joint continuity of their local times follow from the 
results of \cite{B73, P78, GH80}; see also \cite{X09}.

(iii). By Theorem \ref{thm:LT_E}, $\{u(t, x) : (t, x) \in T\}$ has a square-integrable 
local time on $T$. Denote this local time by $L(v, T)$.
In particular, a.s., for all $B \in \mathscr{B}(\R^d)$ and all $S \in \mathscr{B}(T)$,
\begin{equation}\label{LT:def-eq}
\lambda_{1+N}\{ (t, x) \in S : u(t, x) \in B \} = \int_B L(v, S)\, dv.
\end{equation}
We need to find a version $L^*$ of the local time 
that is jointly continuous.
Let $Q_z = (-\infty, z] \cap T$ for $z \in T$.
In what follows, we can assume that $T$ has side lengths $\le r_0$ 
so that Proposition \ref{prop:M-LT} applies, because
it is enough to prove existence of jointly continuous local times on 
sufficiently small subintervals of $T$.
For all even numbers $n \ge 2$, for all $v, \tilde v\in \R^d$, for all $z, \tilde z \in T$, 
we have
\begin{equation*}
\E[(L(v, Q_z) - L(\tilde v, Q_{\tilde z}))^n] 
\le 2^{n-1}\{ \E[(L(v, Q_z) - L(v, Q_{\tilde z}))^n] 
+ \E[(L(v, Q_{\tilde z}) - L(\tilde v, Q_{\tilde z}))^n] \}.
\end{equation*}
For the first term, 
the difference $L(v, Q_z) - L(v, Q_{\tilde z})$ can be written as
a finite sum of terms (the number of which depends only on $N$) 
of the form $L(v, I_j)$,
where each $I_j$ is a subinterval of $T$ with at least one of its side length 
$\le |z - \tilde z|$, so we can use 
Proposition \ref{prop:M-LT}(i) to bound this term.
Also, we can bound the second term by Proposition \ref{prop:M-LT}(ii).
Hence, we have
\[ \E[(L(v, Q_z) - L(\tilde v, Q_{\tilde z}))^n] \le 
C_n (|v - \tilde v|^\gamma + |z - \tilde z|^\delta)^n \]
with constants $0 < \gamma < \min\{\frac 1 2 (\frac{1+N}{\alpha}-d), 1\}$ and 
$\delta = 1-\frac{\alpha d}{1+N}$.
Then, by a multiparameter version of Kolmogorov's continuity theorem 
(\cite[Proposition 4.2]{CD14} or \cite[Theorem 1.4.1]{Ku}),
we can obtain a process 
$\{ L^*(v, Q_z) : v \in \R^d, z \in T \}$ such that
$(v, z) \mapsto L^*(v, Q_z)$ is jointly continuous 
(moreover, locally H\"older continuous of order $< \gamma$ in $v$, 
and of order $< \delta$ in $z$) and 
$\P\{ L^*(v, Q_z) = L(v, Q_z) \} = 1$ for every $v \in \R^d$ and $z \in T$.
To verify that this version $L^*$ is still a local time, 
note that, for each $z \in T$, by Fubini's theorem, 
\begin{equation*}
\int_\Omega \lambda_d\{ v : L^*(v, Q_z) \ne L(v, Q_z) \} \,d\P 
= \int_{\R^d} \P\{ \omega : L^*(v, Q_z) \ne L(v, Q_z)\} \, dv= 0.
\end{equation*}
Then, there is a single event of probability 1 on which
for all rational $z \in T$ simultaneously, we have
$L^*(v, Q_z) = L(v, Q_z)$ a.e.\ $v$.
This and \eqref{LT:def-eq} imply that $L^*$ satisfies 
\begin{equation*}
\lambda_{1+N}\{ (t, x) \in Q_z : u(t, x) \in B \} = \int_B L^*(v, Q_z)\, dv \quad \text{a.s.}
\end{equation*}
for all rational $z \in T$,
and hence for all $z \in T$ by the continuity of $z \mapsto L^*(v, Q_z)$.
This proves that $L^*$ is a local time and finishes the proof.
\end{proof}

\medskip
\section{Regularity of the local times}

In this section, we investigate the regularity of the local times of $u(t, x)$.
As we have seen, the processes $t \mapsto u(t, x)$ and $x \mapsto u(t, x)$ satisfy 
the strong LND property, so the results of Xiao \cite{X97, X08} can be 
applied to obtain moment estimates, H\"older conditions and moduli of continuity 
for the respective local times.
In the following, we will treat the local times of $u$ regarded
as the random field $(t, x) \mapsto u(t, x)$.
First, we study the differentiability of the local time $L(v, T)$ in $v$ 
and the H\"older regularity of its derivatives.
Then, we give a result on the local and uniform moduli of continuity of $L$ in the set 
variable $T$ and discuss their implications on sample function oscillations.

Conditions for local times of Gaussian random fields to have square-integrable 
partial derivatives have been given in \cite{B69'} and \cite[\S 28]{GH80}.
In Theorem \ref{thm:D} below, we obtain conditions for the existence of 
continuous partial derivatives in $v$.
We make use of the Fourier representation \eqref{LT:L2}
and the estimates in Lemma \ref{lem:J} above, 
which have been established using the spherical integral form of strong LND, 
to provide sufficient conditions for the local times of $u$ to have partial derivatives 
(up to certain order) that are jointly continuous and H\"older continuous.
We employ the Fourier analytic approach of \cite{E81} to prove this result.

\begin{theorem}\label{thm:D}
Assume \eqref{beta}.
Let $T$ be a compact interval in $(0, \infty) \times \R^N$.
If $\alpha (d + 2K) < 1+N$ for some integer $K \ge 1$,
then $u = \{u(t, x) : (t, x) \in T\}$ has a local time $L(v, T)$ such that 
all of its partial derivatives
\[ \partial^p L(v, T) = 
\frac{\partial^{p_1} \cdots \partial^{p_d}}{\partial v_1^{p_1} \cdots \partial v_d^{p_d}} 
L(v, T)\]
of order $|p| \le K$ exist and
are a.s.\ jointly continuous and locally H\"older continuous in $v$ 
of any exponent $\gamma < \min\{\frac 1 2 (\frac{1+N}{\alpha} - d - 2|p|), 1\}$.
\end{theorem}

\begin{proof}
By Theorem \ref{thm:JC}, $u$ has a jointly continuous local time $L(\cdot, I)$
on any interval $I \subset T$.
Also, according to the paragraph preceding Theorem \ref{thm:LT_E}, 
$L(\cdot, I)$ is a.e.\ equal to the inverse $L^2$-Fourier transform of $\hat\nu_I$, 
which can be expressed as 
the limit of $L_M(\cdot, I)$ in $L^2(\R^d)$ as $M \to \infty$, where
\[ L_M(v, I) = 
(2\pi)^{-d} \int_{[-M, M]^d} e^{-i\xi \cdot v} \int_I e^{i\xi \cdot u(z)} dz\,d\xi,
\quad v \in \R^d. \]
For any multi-index $p = (p_1, \dots, p_d)$ with $|p| = p_1 + \dots + p_d \le K$,
\[ \partial^p L_M(v, I) = 
(2\pi)^{-d} \int_{[-M, M]^d} \prod_{k=1}^d (-i\xi_k)^{p_k} 
e^{-i\xi \cdot v} \int_I e^{i\xi \cdot u(z)} dz\,d\xi. \]
Let $n = 2m > 0$ be an even number.
We are going to show that, for small subintervals $I$ of $T$, 
$\partial^p L(v, I)$ exists and
\begin{equation}\label{lim:dL}
\sup_{v \in \R^d} \E(|\partial^p L_M(v, I) - \partial^p L(v, I)|^{2m}) \to 0 \quad 
\text{as } M \to \infty.
\end{equation}
Indeed, note that $\partial^p L_M(v, I)$ is real-valued, and for $0 \le M < M'$,
\begin{equation*}
\begin{split}
&\E((\partial^p L_{M'}(v, I) - \partial^p L_M(v, I))^{2m})\\
&=(2\pi)^{-2md}\, \E\int_{([-M', M']^d\setminus [-M, M]^d)^{2m}} 
\prod_{j=1}^{2m} \bigg(\prod_{k=1}^d (-i\xi^j_k)^{p_k} \bigg)e^{-i\xi^j \cdot v}
\int_{I^{2m}} e^{i\sum_{j=1}^{2m}\xi^j \cdot u(z^j)}
\,d\bar z\, d\bar \xi\\
&\le (2\pi)^{-2md}\, \int_{([-M', M']^d\setminus[-M, M]^d)^{2m}} \int_{I^{2m}} 
\bigg(\prod_{j=1}^{2m} \prod_{k=1}^d {|\xi^j_k|}^{p_k}\bigg) 
\E\big(e^{i\sum_{j=1}^{2m}\xi^j \cdot u(z^j)}\big) \, d\bar z \, d\bar \xi,
\end{split}
\end{equation*}
uniformly in $v$.
By Lemma \ref{lem:J} with $q_{j,k} = p_k$ and $q = |p|$,
\begin{equation}
\int_{\R^{2md}} \int_{I^{2m}} \bigg(\prod_{j=1}^{2m} \prod_{k=1}^d {|\xi^j_k|}^{p_k}
\bigg) \E\big(e^{i\sum_{j=1}^{2m}\xi^j \cdot u(z^j)}\big)\, d\bar z \, d\bar \xi < \infty.
\end{equation}
Then, by the dominated convergence theorem, as $M \to \infty$, 
$\partial^p L_M(v, I)$ converges to a limit, denoted by $X_p(v, I)$, in $L^{2m}(\P)$, 
uniformly in $v$.
In particular, we can extract a subsequence $M_j \to \infty$ such that 
\[\sup_{v \in \R^d} \E(|\partial^pL_{M_j}(v, I) - \partial^p L_{M_{j-1}}(v, I)|) 
\le 2^{-j}.\]
For each compact set $F \subset \R^d$, by Fubini's theorem,
\[ \E\int_{F} \sum_{j=2}^\infty |\partial^p L_{M_j}(v, I) - \partial^p L_{M_{j-1}}(v, I)|\, dv 
\le \sum_{j=2}^\infty 2^{-j} \lambda_d(F) < \infty.\]
Hence, by taking a sequence of compact sets $F_i \uparrow \R^d$,
we can find a single event of probability 1 on which 
$\sum_{j=2}^\infty |\partial^p L_{M_j}(v, I) - \partial^p L_{M_{j-1}}(v, I)|$ is 
locally integrable in $v$, so that 
$\partial^p L_{M_j}(v, I) \to X_p(v, I)$ in $L^1_{\mathrm{loc}}(\R^d)$ a.s.
Also, we know that $\int |L_M(v, I) - L(v, I)|^2 dv \to 0$ a.s.
These imply that, on an event of probability 1,
for all smooth test functions $\phi(v)$ with compact support, 
\begin{equation*}
\begin{split}
\int_{\R^d} X_p(v, I)\, \phi(v)\, dv 
&= \lim_{j \to \infty}\int_{\R^d} \partial^p L_{M_j}(v, I)\, \phi(v)\, dv\\ 
&= \lim_{j \to \infty} (-1)^{|p|} \int_{\R^d} L_{M_j}(v, I)\, \partial^p \phi(v)\, dv\\
&= (-1)^{|p|}  \int_{\R^d} L(v, I)\, \partial^p \phi(v)\, dv.
\end{split}
\end{equation*}
This proves, for each $|p|\le K$, the existence of 
$\partial^p L(v, I) (= X_p(v, I))$ as a weak derivative of $L(v, I)$,
which satisfies \eqref{lim:dL}.
The existence of (jointly) continuous derivatives and their H\"older continuity
will follow from an application of Kolmogorov's continuity theorem as in Theorem \ref{thm:JC}
once we show that there is some constant $0 < \delta < 1$ such that 
for all sufficiently small intervals $I \subset T$,
\begin{align}\label{M:dL}
\E(|\partial^p L(v, I)|^n) &\le C_n {[\lambda_{1+N}(I)]}^{n\delta},
\end{align}
and for any constant $0 < \gamma < \min\{ \frac 1 2 (\frac{1+N}{\alpha} - d - 2|p|), 1\}$,
\begin{align}
\label{MI:dL}
\E(|\partial^p L(v, I) - \partial^p L(\tilde v, I)|^n)
&\le C'_n |v - \tilde v|^{n\gamma},
\end{align}
where $C_n$ and $C'_n$ are constants depending on $n = 2m$ (and also on $\gamma$ for $C'_n$), but not on $v, \tilde v$ or $I$.
In fact, \eqref{M:dL} follows from
\begin{equation*}
\begin{split}
\E(|\partial^p L(v, I)|^{2m}) 
& = \lim_{M \to \infty} \E((\partial^p L_M(v, I))^{2m})\\ 
&\le C^{2m} \int_{\R^{2md}} \int_{I^{2m}} 
\bigg(\prod_{j=1}^{2m} \prod_{k=1}^d {|\xi^j_k|}^{p_k}\bigg) 
\E\big(e^{i\sum_{j=1}^{2m}\xi^j \cdot u(z^j)}\big)\, d\bar z \, d\bar \xi 
\end{split}
\end{equation*}
and Lemma \ref{lem:J}(i) with $q_{j, k} = p_k$ and $q = |p|$, which yields
$\delta = (1 - \frac{\alpha(d+2|p|)}{1+N})\frac{d}{d+2|p|}$.
As for \eqref{MI:dL}, use the inequality 
$|e^{-ix} - e^{-iy}| \le 2|x-y|^\gamma$ for $0 < \gamma < 1$ to get that
\begin{equation*}
\begin{split}
&\E((\partial^p L(v, I) - \partial^p L(\tilde v, I))^{2m})\\
& \le C^{2m} \int_{\R^{2md}} \int_{I^{2m}} 
\bigg( \prod_{j=1}^{2m} \prod_{k=1}^d |\xi^j_k|^{p_k}\bigg)
\bigg(\prod_{j=1}^{2m}\big|e^{-i\xi^j\cdot v} - e^{-i\xi^j\cdot \tilde v}\big|\bigg)
\E\big(e^{i\sum_{j=1}^{2m} \xi^j\cdot u(z^j)}\big) \, d\bar z\, d\bar\xi\\
& \le C^{2m} |v - \tilde v|^{2m\gamma} \sum_{(k_1, \dots, k_{2m})}
\int_{\R^{2md}} \int_{I^{2m}} 
\bigg( \prod_{j=1}^{2m}\prod_{k=1}^d |\xi^j_k|^{p_k} \bigg)
\bigg( \prod_{j=1}^{2m} |\xi^j_{k_j}|^{\gamma} \bigg)
\E\big(e^{i\sum_{j=1}^{2m} \xi^j\cdot u(z^j)}\big) \, d\bar z\, d\bar\xi,
\end{split}
\end{equation*}
where the sum is taken over all $(k_1, \dots, k_{2m}) \in \{1, \dots, d\}^{2m}$.
Then, estimate the integral term using Lemma \ref{lem:J}
with $q_{j, k} = p_k + \gamma$ if $k = k_j$, and $q_{j, k} = p_k$ otherwise, 
to finish the proof.
\end{proof}

\smallskip

Next, we study the regularity of the local time in the set variable.
The following definition can be found in \cite[p.227]{A}. Let $0 < \gamma \le 1$.
We say that the local time $L$ satisfies a uniform H\"older condition of order $\gamma$
in the set variable if there exists $C< \infty$ such that
for all $v \in \R^d$ and all cubes $I \subset T$ with a sufficiently small side length, 
we have
\[ L(v, I) \le C [\lambda(I)]^\gamma, \]
where $\lambda(I)$ is the Lebesgue measure of the cube $I$.
H\"older conditions of the local times of random fields 
contain rich information about irregularity properties of the sample paths; 
see \cite{A, B72, GH80}.

In fact, we are going to present a result (Theorem \ref{MC:L*}) 
which provides not only information about H\"older condition but also 
the moduli of continuity of the local times in the set variable.
To establish this result, we need the following lemma.

\begin{lemma}\label{lem3}
Assume \eqref{beta} and $\alpha d < 1+N$.
Let $T$ be a compact interval in $(0, \infty)\times \R^N$.
Then the following hold for some constant $r_0 > 0$:
\begin{enumerate}
\item[(i)] For any $b > 0$, there exists a finite constant $c$ such that 
for all $(t, x) \in T \cup \{ 0 \}$ and intervals $I$ in $T$ with side lengths $= r \le r_0$, 
for all $v \in \R^d$ and $A > 1$,
\begin{equation}
\P\Big\{ L(v + u(t, x), I) \ge c \, r^{1+N-\alpha d} A^{\alpha d}  \Big\} \le \exp(-bA).
\end{equation}
\item[(ii)] For any $b > 0$ and 
$0 < \gamma < \min\{\frac 1 2 (\frac{1+N}{\alpha} - d), 1\}$, 
there exists a finite constant $c$ such that for all $(t, x) \in T \cup \{ 0 \}$,
for all intervals $I$ in $T$ with side lengths $= r \le r_0$, 
for all $v, \tilde v \in \R^d$ and $A > 1$,
\begin{equation}
\begin{split}
\P\Big\{ |L(v + u(t, x)&, I) - L(\tilde v + u(t, x), I)|\\
& \ge c\, |v - \tilde v|^\gamma r^{1+N-\alpha d - \alpha \gamma} 
A^{\alpha d + (\frac 1 2 +\alpha)\gamma} \Big\} \le \exp(-bA).
\end{split}
\end{equation}
\end{enumerate}
\end{lemma}

\begin{proof}
For $(t, x) = 0$, note that $u(t, x) = 0$. In this case, 
(i) and (ii) can be proved by the moment estimates in 
Proposition \ref{prop:M-LT} with the use of Chebyshev's inequality and Stirling's formula.
For $(t, x) = (t^0, x^0) \in T$, consider the process 
$\tilde u = \{ \tilde u(s, y) := u(s, y) - u(t^0, x^0) : (s, y) \in T\}$,
so that the local time $\tilde L$ of $\tilde u$ exists and 
satisfies $\tilde L(v, I) = L(v + u(t^0, x^0), I)$.
By Proposition \ref{prop:LND}, the conditional variance of $\tilde u$ satisfies
\begin{align*}
&\mathrm{Var}(\tilde u(t, x)|\tilde u(t^1, x^1), \dots, \tilde u(t^{n}, x^{n}))\\
& \ge \mathrm{Var}(u(t, x)|u(t^0, x^0), u(t^1, x^1), \dots, u(t^{n}, x^{n}))\\
& \ge C \int_{\mathbb{S}^{N-1}}\, \min_{0\le i \le n} |(t-t^i) + (x - x^i) \cdot w|^{2\alpha} \, \sigma(dw).
\end{align*}
With slight modifications, the proof of Proposition \ref{prop:M-LT} can be carried 
over to $\tilde u$ and $\tilde L$ to yield
\begin{gather*}
\E[L(v + u(t, x), I)^n] \le C^n (n!)^{\alpha d} r^{n(1+N-\alpha d)},\\
\E[(L(v + u(t, x), I) - L(\tilde v + u(t, x), I))^n] \le 
C^n |v - \tilde v|^{n\gamma} (n!)^{\alpha d + (\frac 1 2 +\alpha)\gamma} 
r^{n(1+N - \alpha d - \alpha \gamma)}
\end{gather*}
for $\gamma < \min\{\frac 1 2 (\frac{1+N}{\alpha}-d), 1\}$ and $n\ge 2$ even.
These estimates imply (i) and (ii) as in the first part the proof.
\end{proof}

From this, we can deduce the local and uniform moduli of continuity 
of the local times in the set variable.
Let $I_r(t, x) = [t-r, t+r] \times \prod_{\ell=1}^N [x_\ell - r, x_\ell+r]$, and let
$\mathscr{I}(T, r)$ denote the set of all intervals $I_\rho(t, x)$ in $T$ 
with $\rho \le r$.

\begin{theorem}\label{MC:L*}
Assume \eqref{beta} and $\alpha d < 1+N$.
For any compact interval $T$ in $(0, \infty) \times \R^N$,
there exists a finite constant $C_1$ such that for any fixed $(t, x) \in T$,
\begin{equation}\label{MCLT1}
\limsup_{r \to 0} \sup_{v \in \R^d}\,\frac{L(v, I_r(t, x))}
{r^{1+N - \alpha d} (\log\log(1/r))^{\alpha d}} \le C_1 \quad \text{a.s.}
\end{equation}
Moreover, there exists a finite constant $C_2$ such that
\begin{equation}\label{MCLT2}
\limsup_{r \to 0} \sup_{I \in \mathscr{I}(T, r)} \sup_{v \in \R^d}\,\frac{L(v, I)}
{[\lambda_{1+N}(I)]^{1- \frac{\alpha d}{1+N}} {(\log(1/\lambda_{1+N}(I)))}^{\alpha d}} 
\le C_2 \quad \text{a.s.}
\end{equation}
Hence, $L$ satisfies a uniform H\"older condition of order $\gamma$ in the set variable
for any $0 < \gamma < 1-\frac{\alpha d}{1+N}$.
\end{theorem}

\begin{remark}
By the strong LND property and the results of Xiao \cite{X97, X08}, 
the local and uniform moduli of continuity for the 
local times of $t \mapsto u(t, x_0)$ and $x \mapsto u(t_0, x)$ are as follows
in comparison to the above theorem.
For $t \mapsto u(t, x_0)$ and $\alpha d < 1$, those are
\[ r^{1-\alpha d}(\log\log(1/r))^{\alpha d} \quad \text{and} \quad
[\lambda_1(I)]^{1-\alpha d}(\log(1/\lambda_1(I)))^{\alpha d}.\]
For $x \mapsto u(t_0, x)$ and $\alpha d < N$, those are
\[ r^{N-\alpha d}(\log\log(1/r))^{\frac{\alpha d}{N}} \quad \text{and} \quad
[\lambda_N(I)]^{1-\frac{\alpha d}{N}}(\log(1/\lambda_N(I)))^{\frac{\alpha d}{N}}.\]
\end{remark}

\begin{proof}[Proof of Theorem \ref{MC:L*}]
The proof of this theorem is based on Lemma \ref{lem3} and
a chaining argument as in \cite{E81, X97}. 
We will make some necessary changes for our purposes.

Let $L^*(I) = \sup\{ L(v, I) : v \in \R^d\}$ and 
$\varphi(r) = r^{1+N - \alpha d} (\log\log(1/r))^{\alpha d}$.
To prove \eqref{MCLT1}, it suffices to prove that for any fixed $(t, x) \in I$, 
\begin{equation}\label{L*}
\limsup_{n \to \infty} \frac{L^*(J_n)}{\varphi(2^{-n})} \le C_1 \quad \text{a.s.},
\end{equation}
where $J_n = I_{2^{-n}}(t, x)$. The proof of \eqref{L*} is divided into four steps.

Step 1.
Let $\delta_n = c_0\, 2^{-n\alpha} \sqrt{\log n}$.
By Lemma 2.1 of Talagrand \cite{T95}, there exists a constant $c_0< \infty$
such that for $n$ large,
\begin{equation*}
\P\bigg\{ \sup_{(s, y) \in J_n} |u(s, y) - u(t, x)| \ge \delta_n \bigg\} \le n^{-2},
\end{equation*}
which, by the Borel--Cantelli lemma, implies that with probability 1, for all $n$ large,
\begin{equation}\label{supu}
\sup_{(s, y) \in J_n} |u(s, y) - u(t, x)| \le \delta_n.
\end{equation}

Step 2. Let $\theta_n = 2^{-n\alpha} (\log n)^{-\frac 1 2 -\alpha}$ and
\begin{equation*}
G_n = \Big\{ v \in \R^d: 
|v| \le \delta_n \text{ and } v = \theta_n p \text{ for some } p \in \Z^d \Big\}.
\end{equation*}
The cardinality of $G_n$ is $\le C (\log n)^{(1 +\alpha)d}$.
By Lemma \ref{lem3}(i) with $b = 2$, 
we can find a finite constant $c$ such that for all $n$ large
\begin{equation*}
\P\bigg\{ \max_{v \in G_n} L(v+u(t, x), J_n) \ge c\, \varphi(2^{-n}) \bigg\} 
\le C (\log n)^{(1 +\alpha)d}\, n^{-2}.
\end{equation*}
It follows from the Borel--Cantelli lemma that with probability 1, for all $n$ large,
\begin{equation}\label{maxL}
\max_{v \in G_n} L(v + u(t, x), J_n) \le c\, \varphi(2^{-n}).
\end{equation}

Step 3. For $n, k \ge 1$ and $v \in G_n$, define
\begin{equation*}
F(n, k, v) = \Big\{ y \in \R^d : y = v + \theta_n \sum_{j=1}^k \eps_j 2^{-j},
\eps_j \in \{0, 1\}^d \text{ for } 1\le j \le k \Big\}.
\end{equation*}
We say that a pair $y_1, y_2 \in F(n, k, v)$ is linked 
if $y_1 - y_2 = \theta_n \eps 2^{-k}$ for some $\eps \in \{0, 1\}^d$.
Fix any $0 < \gamma < \min\{\frac 1 2 (\frac{1+N}{\alpha} - d), 1\}$.
Consider the event $A_n$ defined by
\begin{equation*}
\begin{split}
A_n = \bigcup_{v \in G_n} \bigcup_{k \ge 1} \bigcup_{y_1\sim y_2} 
\Big\{ &|L(y_1 + u(t, x), J_n) - L(y_2 + u(t, x), J_n)|\\
& \quad \ge c\, 2^{-n(1+N - \alpha d - \alpha\gamma)} |y_1 - y_2|^{\gamma}
(k\log n)^{\alpha d + (\frac1 2+\alpha) \gamma} \Big\},
\end{split}
\end{equation*}
where $\bigcup_{y_1\sim y_2}$ denotes 
the union over all linked pairs in $y_1, y_2 \in F(n, k, v)$. 
There are at most $2^{kd} 3^d$ linked pairs in $F(n, k, v)$.
Then by Lemma \ref{lem3}(ii) with $b = 2$, for $n$ large, 
\begin{equation*}
\begin{split}
\P(A_n) &\le C (\log n)^{(1+\alpha)d} \sum_{k=1}^\infty 2^{kd} \exp(-2k\log n)\\
& = C (\log n)^{(1+\alpha)d} \frac{2^d n^{-2}}{1-2^d n^{-2}},
\end{split}
\end{equation*}
so $\sum_{n=1}^\infty \P(A_n) < \infty$. By the Borel--Cantelli lemma,
a.s.\ $A_n$ occurs at most finitely many times.

Step 4. We proceed with the chaining argument in \cite{E81, X97}.
For $y \in \R^d$ with $|y| \le \delta_n$, we can represent $y$ 
as the limit of $(y_k)$, where
\begin{equation*}
y_k = v + \theta_n \sum_{j=1}^k \eps_j 2^{-j},
\end{equation*}
$y_0 := v \in G_n$ and $\eps_j \in \{0, 1\}^d$ for $j = 1, \dots, k$.
Since $L(v, J_n)$ is continuous in $v$, we see that on the event $A_n^c$,
\begin{equation}\label{L-L}
\begin{split}
&|L(y + u(t, x), J_n) - L(v + u(t, x), J_n)|\\
& \le \sum_{k=1}^\infty |L(y_k + u(t, x), J_n) - L(y_{k-1} + u(t, x), J_n)|\\
& \le \sum_{k=1}^\infty c\, 2^{-n(1+N-\alpha d - \alpha \gamma)}
{(\theta_n 2^{-k})}^\gamma {(k\log n)}^{\alpha d + (\frac 1 2+\alpha)\gamma}\\
& \le c \, 2^{-n(1+N-\alpha d)} (\log n)^{\alpha d}
\sum_{k=1}^\infty k^{\alpha d + (\frac 1 2+\alpha)\gamma} 2^{-k\gamma}\\
& \le C \varphi(2^{-n}).
\end{split}
\end{equation}
Now, \eqref{maxL} and \eqref{L-L} imply that a.s.\ for all $n$ large, 
\begin{equation*}
\sup_{|y| \le \delta_n} L(y + u(t, x), J_n) \le C \varphi(2^{-n}).
\end{equation*}
Since $L(\cdot, J_n)$ is supported on $\overline{u(J_n)}$, 
this, together with \eqref{supu}, implies \eqref{L*} and hence \eqref{MCLT1}.

The proof of \eqref{MCLT2} is similar. 
Let $\Phi(r) = r^{1- \frac{\alpha d}{1+N}}(\log(1/r))^{\alpha d}$
and $\mathscr{D}_n$ denote the collection of dyadic cubes 
$\prod_{j=1}^{1+N}[i_j2^{-n}, (i_j+1)2^{-n}]$ that intersect with $T$, 
where $i_j \in \Z$.
Note that any interval $I = I_r(t, x)$ with $r \le 2^{-n}$ can be 
covered by at most $8^{1+N}$ dyadic intervals $D_i$ of side length $\le r$
such that each $D_i$ is in $\bigcup_{m \ge n} \mathscr{D}_m$ and
has Lebesgue measure $\lambda_{1+N}(D_i) \le \lambda_{1+N}(I)$.
Then
\[L^*(I) \le 8^{1+N} \sup_{m \ge n} \sup_{D \in \mathscr{D}_m} L^*(D).\]
Also, $\Phi(r)$ is increasing for $r > 0$ small, so it suffices to prove that
\begin{equation}\label{supL*}
\limsup_{n \to \infty} \max_{D \in \mathscr{D}_n} \frac{L^*(D)}{\Phi(\lambda_{1+N}(D))} 
\le C_2 \quad \text{a.s.}
\end{equation}
To this end, define $\theta_n = 2^{-n\alpha}(\log 2^n)^{-\frac 1 2 -\alpha}$ and
\[ G_n = 
\Big\{ v \in \R^d : |v| \le n \text{ and } v = \theta_n p \text{ for some } p \in \Z^d \Big\}. \]
By Lemma \ref{lem3}(i), we can find a large enough $C$ so that 
a.s.\ for all $n$ large,
\begin{equation}\label{maxL*}
\max_{D \in \mathscr{D}_n} \max_{v \in G_n} L(v, D) \le C \Phi(\lambda_{1+N}(D)) 
\quad \text{a.s.}
\end{equation}
Define $F(n, k, v)$ as in the the proof of \eqref{MCLT1} above, and similarly, let
\begin{align*}
A_n = 
\bigcup_{D \in \mathscr{D}_n} \bigcup_{v \in G_n} \bigcup_{k \ge 1} 
\bigcup_{y_1\sim y_2} \Big\{ & |L(y_1, D) - L(y_2, D)| \\
& \quad \ge C 2^{-n(1+N - \alpha d - \alpha \gamma)}|y_1 - y_2|^\gamma
(k \log 2^n)^{\alpha d + (\frac 1 2+\alpha) \gamma}\Big\}.
\end{align*}
Then by Lemma \ref{lem3}(ii), a.s.\ $A_n$ occurs at most finitely many times. 
Since $u(t, x)$ is continuous, there exists $n_0 = n_0(\omega)$ such that
$\sup_{(t, x) \in T} |u(t, x)| \le n_0$ a.s.
If $|y| \le n$, then by the chaining argument as before, we can deduce that on 
the event $A_n^c$,
\begin{align*}
|L(y, D) - L(v, D)| &\le C 2^{-n(1+N-\alpha d)} (\log 2^n)^{\alpha d}\\
& \le C \Phi(\lambda_{1+N}(D)).
\end{align*}
Then by \eqref{maxL*}, we see that a.s.\ for all $n$ large,
\begin{equation}\label{maxL*2}
\max_{D \in \mathscr{D}_n} \sup_{|y|\le n} L(y, D) \le C \Phi(\lambda_{1+N}(D)).
\end{equation}
If $|y| > n_0$, then $y \not\in \overline{u(T)}$, thus $L(y, D) = 0$.
This together with \eqref{maxL*2} implies \eqref{supL*}, and 
hence completes the proof of \eqref{MCLT2}.
\end{proof}

As discussed in \cite{E81}, 
the moduli of continuity of the local times are closely related to 
the degree of oscillations of the sample functions.
The former leads to lower envelops for the oscillations.
A consequence of this is that the solution is nowhere differentiable.

\begin{theorem}\label{thm:ND}
Assume \eqref{beta}.
Let $T$ be a compact interval in $(0, \infty) \times \R^N$.
Then there exist positive constants $C_3$ and $C_4$ such that 
for each fixed $(t, x) \in T$,
\begin{equation}\label{CLIL}
\liminf_{r \to 0} \sup_{(s, y) \in I_r(t, x)} \frac{|u(s, y) - u(t, x)|}{r^\alpha(\log\log(1/r))^{-\alpha}} \ge C_3
\quad \text{a.s.}
\end{equation}
and
\begin{equation}\label{MND}
\liminf_{r \to 0} \inf_{(t, x) \in T} \sup_{(s, y) \in I_r(t, x)} \frac{|u(s, y) - u(t, x)|}{r^\alpha(\log(1/r))^{-\alpha}} \ge C_4 
\quad \text{a.s.}
\end{equation}
In particular, \eqref{MND} implies that 
$(t, x) \mapsto u(t, x)$ is a.s.\ nowhere differentiable on $(0, \infty) \times \R^N$.
\end{theorem}

\begin{proof}[Proof of Theorem \ref{thm:ND}]
Since $u$ has i.i.d.\ components, it suffices to prove the result for $d = 1$.
Then $\alpha d = \alpha < 1+N$, so that $u(t, x)$ has a jointly 
continuous local time $L(v, T)$, $v \in \R$.
Fix $(t, x) \in T$ and $r > 0$. 
Since $L(v, I_r(t, x)) = 0$ for $v \not\in \overline{u(I_r(t, x))}$, we have
\begin{equation}\label{LTineq}
\begin{split}
\lambda_{1+N}(I_r(t, x)) & = \int_{\, \overline{u(I_r(t, x))}} L(v, I_r(t, x))\, dv\\
&\le 2 \sup_{(s, y) \in I_r(t, x)} |u(s, y) - u(t, x)| \times \sup_{v \in \R}L(v, I_r(t, x)) .
\end{split}
\end{equation}
Therefore, \eqref{CLIL} follows from \eqref{LTineq} and \eqref{MCLT1},
while \eqref{MND} follows from \eqref{LTineq} and \eqref{MCLT2}.
\end{proof}

The corresponding results for $t \mapsto u(t, x_0)$ and $x \mapsto u(t_0, x)$ are:
\[\liminf_{r \to 0} \sup_{s \in I_r(t)} \frac{|u(s, x_0) - u(t, x_0)|}
{r^\alpha(\log\log(1/r))^{-\alpha}} \ge C, \quad
\liminf_{r \to 0} \inf_{t \in T_1} \sup_{s \in I_r(t)} \frac{|u(s, x_0) - u(t, x_0)|}
{r^\alpha(\log(1/r))^{-\alpha}} \ge C, \]
\[\liminf_{r \to 0} \sup_{y \in I_r(x)} \frac{|u(t_0, y) - u(t_0, x)|}
{r^\alpha(\log\log(1/r))^{-\alpha/N}} \ge C, \quad
\liminf_{r \to 0} \inf_{x \in T_2} \sup_{y \in I_r(x)} \frac{|u(t_0, y) - u(t_0, x)|}
{r^\alpha(\log(1/r))^{-\alpha/N}} \ge C, \]
which follow from the strong LND property and the results of Xiao \cite{X97, X08}.
In particular, this indicates that \eqref{CLIL} is not sharp when $N \ge 2$.
It would be interesting to derive sharp envelops for \eqref{CLIL} and \eqref{MND}
so that the $\liminf$s are positive and finite.
In fact, when $N=1$ and $\dot{W}(t, x)$ is white in time ($H=1/2$) and 
has spatial covariance given by the Riesz kernel, 
the following Chung-type law of the iterated logarithm is proved in \cite{L21}: 
\[ \liminf_{r \to 0} \sup_{(s, y) \in I_r(t, x)} 
\frac{|u(s, y) - u(t, x)|}{r^\alpha(\log\log(1/r))^{-\alpha}} = C \quad \text{a.s.}\]
where $C$ is a positive finite constant.

\end{document}